\documentclass[12pt]{article}
\usepackage{amsmath,amssymb,amsthm,amsfonts,amscd,euscript,verbatim, t1enc, newlfont, xypic} 
\usepackage{hyperref}
\hypersetup{breaklinks}

 \theoremstyle{plain}
\newtheorem{theo}{Theorem}[subsection]

\newtheorem{pr}[theo]{Proposition}

 \newtheorem{coro}[theo]{Corollary}
  
\theoremstyle{remark}
\newtheorem{rema}[theo]{Remark}

\theoremstyle{definition}
\newtheorem{defi}[theo]{Definition}

 \newcommand\lan{\langle}
\newcommand\ra{\rangle}

\newcommand\gd{\mathfrak{D}}
\newcommand\gdos{\mathfrak{D}(S)}
\newcommand\gdeos{\mathfrak{D}^{eff}(S)}

\newcommand\wchow{{w_{Chow}}}

\newcommand\dmes{\dm^{eff}_{\delta}(S)}
\newcommand\dmgm{DM_{gm}}
\newcommand\dme{DM_-^{eff}{}}
\newcommand\dm{DM}
\newcommand\dms{DM(S)}
\newcommand\dmsp{DM(S')}

\newcommand\pdmcs{Pro-DM_c(S)}

\newcommand\dmcs{DM_c(S)}
\newcommand\dmcsp{DM_c(S')}
\newcommand\dmcspp{DM_c(S'')}

\newcommand\dmx{DM(X)}
\newcommand\dmy{DM(Y)}

\newcommand\dmcx{DM_c(X)}
\newcommand\dmcy{DM_c(Y)}

\newcommand\dmc{DM_c}
\newcommand\dmu{DM(U)}

\newcommand\mg{\mathcal{M}}

\newcommand\mgx{\mathcal{M}_X}
\newcommand\mgxp{\mathcal{M}_{X'}}
\newcommand\mgs{\mathcal{M}_S}
\newcommand\mgy{\mathcal{M}_Y}
\newcommand\mgbm{\mathcal{M}^{BM}_S}
\newcommand\mgbms{\mathcal{M}^{BM}_S}

\newcommand\mgbmx{\mathcal{M}^{BM}_X}

\newcommand\chows{Chow(S)}
\newcommand\chowes{Chow^{eff}(S)}
\newcommand\wger{w_{Ger}}
\newcommand\wgers{w_{Ger}(S)}
\newcommand\hwgers{\hw_{Ger}(S)}

\newcommand\obj{Obj}

\newcommand\id{id}
\newcommand\cu{\underline{C}}
\newcommand\du{\underline{D}}
\newcommand\au{\underline{A}}
\newcommand\eu{\underline{E}}

\newcommand\z{{\mathbb{Z}}}
\newcommand\q{{\mathbb{Q}}}

\newcommand\p{\mathbb{P}}

\newcommand\codim{\operatorname{codim}}
\newcommand\cl{\operatorname{cl}}

\newcommand\al{\alpha}
\newcommand\be{\beta}

\newcommand\xx{{\mathcal{X}}}
\newcommand\xxpp{{\mathcal{X''}}}

\newcommand\ns{\{0\}}

\DeclareMathOperator\prli{\varprojlim}
\DeclareMathOperator\inli{\varinjlim}

\newcommand\chow{Chow}
\newcommand\chowe{Chow^{eff}}

\newcommand\ab{Ab}
\newcommand\qv{\q-\operatorname{Vect}}

\newcommand\spe{\operatorname{Spec}}

\newcommand\sms{{\mathfrak{Sm}}_S}
\newcommand\smx{{\mathfrak{Sm}}_X}
\newcommand\sm{\mathfrak{Sm}}
\newcommand\smy{\mathfrak{Sm}_Y}
\newcommand\smxp{\mathfrak{Sm}_{X'}}

\newcommand\tcho{t_{Chow}}
\newcommand\thom{t_{hom}}
\newcommand\thome{t_{hom}^{eff}}

\DeclareMathOperator\dimz{\delta}
\DeclareMathOperator\codimz{\operatorname{codim}^{\delta}}
 \DeclareMathOperator\ke{\operatorname{Ker}}
 
\DeclareMathOperator\imm{\operatorname{Im}}

\DeclareMathOperator\kar{\operatorname{Kar}}

\newcommand\hrt{{\underline{Ht}}}
\newcommand\hrthom{{\underline{Ht}_{hom}}}

\newcommand\hw{{\underline{Hw}}}

\begin{document}

 \title{On  perverse homotopy $t$-structures,  coniveau spectral sequences, cycle modules, and relative Gersten weight structures} 
 \author{Mikhail V. Bondarko
   \thanks{ 
 The work is supported by RFBR
(grants no.   14-01-00393 and 15-01-03034), by  Dmitry Zimin's Foundation "Dynasty", and by the Scientific schools grant no. 3856.2014.1.}}\maketitle

\begin{abstract}

\end{abstract}
We construct a certain $t$-structure $\thom=\thom(S)$ for $\dms$ (this is a version of Voevodsky's triangulated category of motives over quite a general base scheme $S$ that was described by Cisinski and D\'eglise); for $S$ being the spectrum of a perfect field this $t$-structure extends the homotopy $t$-structure of Voevodsky, whereas for $S$ being of finite type over (the spectrum of) a discrete valuation ring $\thom(S)$  is the perverse homotopy $t$-structure of Ayoub. 
For any 
("reasonable")  
 scheme $S$ our $t$-structure is characterized in terms of certain stalks of an $H\in \obj \dms$ and its Tate twists at fields over $S$; this generalizes one of the descriptions of the homotopy $t$-structure over a field (given by Voevodsky, Cisinski, and D\'eglise). 
$\thom(S)$ is closely related to certain coniveau spectral sequences for the cohomology of arbitrary
  finite type $S$-schemes (actually, we study the cohomology of  Borel-Moore motives of such schemes). 
  We conjecture that the heart of $\thom(S)$ is given by 
   cycle modules over $S$ 
   (as defined by Rost); for schemes of finite type over characteristic $0$ fields this conjecture was proved by D\'eglise. Our definition of $\thom(S)$ is closely related to  
  a new effectivity filtration for $\dms$ (and for the category of Chow motives $\chows\subset \dms$). We also sketch the construction of a certain Gersten weight structure on the category of $S$-comotives $\gdos\supset \dmcs$; this weight structure yields one more description of $\thom(S)$ and its heart.

\tableofcontents

 \section*{Introduction}
  We construct a certain $t$-structure $\thom=\thom(S)$ for $\dms$ (the category of {\it Beilinson motives}, i.e.,  a version of Voevodsky motivic category over quite a general base scheme $S$ described in \cite{degcis}; 
	we assume that $S$ is 
	{\it reasonable} 
  in the sense described in \S\ref{snotata} below). For $S$ being the spectrum of a perfect field  $\thom(S)$  coincides with the  "stable extension" of the homotopy $t$-structure of Voevodsky that was considered in \S5 of \cite{degmod}, whereas for $S$ being of finite type over %
  (the spectrum of) a discrete valuation ring
	our $t$-structure (essentially) coincides with 
	the perverse homotopy $t$-structure of Ayoub. 

  Recall that one can define the homotopy $t$-structure on Voevodsky's $\dme(k)$ (the category of effective motivic complexes) 
  via two quite distinct methods. The first one relies on the canonical $t$-structure on the derived category of Nisnevich sheaves with transfers; this corresponds to the study of the "stalks" of an object $H$ of $\dme$ at all essentially smooth Henselian schemes over $k$. The second way is to consider the stalks of 
  $H(-i)[-i]$ (considered as an object of some 'stable' category of motivic complexes) at function fields over $k$ (only). 
  It turns out that the second method carries over to motives over quite general bases if one defines the stalks in question "appropriately" (whereas the main result of \cite{aycex} seems to yield that the first method cannot yield a "nice" $t$-structure on relative motives). So, we use the latter approach;
  one of the main ingredients of our construction is the usage of  Borel-Moore motives of finite type $S$-schemes (in particular, we need them for defining the stalks mentioned). The paper also heavily relies on certain coniveau spectral sequences for the cohomology (of the Borel-Moore motives) of $S$-schemes 
 (note in contrast that in \cite{degfib} 'the usual' motives of smooth $S$-schemes are essentially considered). 
   Their construction is more or less automatic given the definition of Borel-Moore motives (that are closely related to the ones considered in \cite{lemm}; see Remark \ref{rnat}(2) below), yet studying their functoriality requires quite a bit of effort.
  These coniveau spectral sequences (for $\dms$-representable cohomology theories) can also be described in terms of $\thom(S)$ (this is a natural analogue of the corresponding results of \cite{ndegl}, \cite{bger}, and \cite{bgern}). In particular, the cohomology of $\mgbm(X)$ for a finite type $X/S$ with the coefficients in an object $H$ of the heart $ \hrthom(S)$ of $\thom(S)$ is just the cohomology of the corresponding Cousin complex. Moreover, we prove 
  for  such an $H$ that the collection of all stalks of $H(i)[i]$ (for all $i\in \z$) at all (spectra of essentially of finite type) fields over $S$ satisfies some of the axioms of cycle modules over $S$ (as defined in \cite{rostc}).
  
  Now we describe the idea that has initiated 
   the 
   work on this paper. 
   The author has realized the following: similarly to the case of motives over a field (cf. \S4.9 of \cite{bger}), the Chow weight structure on
  $\dmcs$ (as constructed in \cite{hebpo} and \cite{brelmot}) should be closely related to the {\it Gersten weight structure} (for a certain triangulated category $\gdos$ of $S$-{\it comotives}). So, this weight structure $\wger(S)$ should  be "cogenerated" by the twists of a certain $\chowes\subset \chows$ (the latter category was defined in \cite{hebpo} and \cite{brelmot}) by $(i)[i]$, whereas   the $t$-structure $\thom(S)$ for $\dms$ that is {\it orthogonal} to this weight structure should be "generated" by $\chowes(i)[i]$. Thus the author gave a definition of $\chowes$ generalizing the category of effective Chow motives over a field (considered as a full additive subcategory of 
  $\dme(k)\subset \dm(\spe k)$). The idea to consider "our" Borel-Moore motives  came from the definition of $\chow(S)$ (only later the author has realized that these motives are closely related to the ones considered in \cite{lemm}). Next he proved Theorem \ref{tmgeff}(II.1) of this paper; together with Proposition \ref{pcss} below (whose proof is more or less straightforward) this result 
  implies  
  the basic properties  $\wger(S)$ and $\thom(S)$. Yet the study of   $\wger(S)$ requires some work on $\gdos$; so the author chose to concentrate on $\thom(S)$ in the current text and only sketch the proofs of the results related to $\wger(S)$ (whereas several  arguments and definitions required for their detailed proofs can be found in \cite{bgern}).
	
	Our perverse homotopy $t$-structures are closely related to the  perverse homotopy $t$-structures studied in \S2.2 of \cite{ayobook}. The latter depend on the choice of some fixed base scheme $S_0$ (all the schemes considered should be of finite type over it); in the case where $S_0$ is the spectrum of a field or a discrete valuation ring several nice properties of these $t$-structures were established in ibid. (including the fact that they differ from our $\thom(-)$ only by $[\dim S_0]$). 
	Moreover, these results recently enabled F. D\'eglise  to prove (in \cite{degchz})  that $\hrthom(S)$ (as well as its integral coefficient version) is equivalent to the (corresponding) category of  $S$-cycle modules 
	 if $S$ is 
	 a variety over a characteristic $0$ field.
	Yet for a general $S$ our results are quite new (since the arguments of \cite{ayobook} heavily rely on a certain existence of alterations statement which 
	is only known to hold over a discrete valuation ring, whereas our approach is 
	somewhat more 'formal').
  
  Lastly we note that it would certainly be very interesting to extend (some of) the results of the current paper to motives with integral coefficients.
  
  Now we describe the contents of the paper; some more information of this sort can be found at the beginnings of sections. 
  
  In \S1 we recall some basics on 
  $t$-structures and motives over a base.
  
  In \S2 we study the Borel-Moore motives of $S$-schemes, their "smooth $S$-models", and coniveau spectral sequences for their cohomology. The most 
  'interesting' statement of the section   is Theorem \ref{tmgeff}(II.1).

  In \S3 we define (our version of) the perverse homotopy $t$-structure $\thom(S)$ and establish its main properties (that generalize the corresponding results of \cite{ayobook}). 
  We prove that the objects of $\hrthom(S)$ satisfy some of the axioms of cycle modules over $S$ (as defined in \cite{rostc}). Next we introduce a certain
effectivity filtration for $\dms$ and study its properties (that are quite similar to the ones over perfect fields). We also sketch the proof of the main properties of a certain Gersten weight structure for the category $\gdos$ of $S$-comotives; these statements yield the motivic functoriality of coniveau spectral sequences and their description in terms of $\thom(S)$.

The author would  like to express his gratitude to   prof. L.E. Positselski and prof. F. D\'eglise for their helpful comments. 
He is also deeply grateful to prof. D\'eglise and to Unit\'e de math\'ematiques pures et appliqu\'ees of
\'Ecole normale sup\'erieure de Lyon for the wonderful working conditions in January of 2015.

\section{Some preliminaries: notation, $t$-structures, and motives over a base}

In this section we 
recall some basics on $t$-structures and relative motives.

In \S\ref{snotata} we introduce some 
notation and definitions. 

In \S\ref{dtst} we
recall the notion of  $t$-structure
(and introduce some notation for it).  We also recall that "compactly generated" {\it preaisles} are  {\it aisles} in the terms of \cite{talosa}.


In \S\ref{sdm}  we recall (following \cite{degcis}) some basic properties of Beilinson motives over  
{\it reasonable}  schemes. 

  \subsection{Some notation and definitions}\label{snotata}

For categories $C,D$ we write 
$D\subset C$ if $D$ is a full 
subcategory of $C$.

 For a category $C,\ X,Y\in\obj C$, we denote by
$C(X,Y)$ the set of  $C$-morphisms from  $X$ to $Y$.
We will say that $X$ is  a {\it
retract} of $Y$ if $\id_X$ can be factored through $Y$. Note that if $C$ is triangulated or abelian 
then $X$ is a  retract of
$Y$ if and only if $X$ is its direct summand.

For any $D\subset C$
the subcategory $D$ is called {\it Karoubi-closed} in $C$ if it
contains all retracts of its objects in $C$. We will call the
smallest Karoubi-closed subcategory of $C$ containing $D$  the {\it
Karoubi-closure} of $D$ in $C$; sometimes we will use the same term
for the class of objects of the Karoubi-closure of a full subcategory
of $C$ (corresponding to some subclass of $\obj C$).

The {\it Karoubization} $\kar(B)$ (no lower index) of an additive
category $B$ is the category of "formal images" of idempotents in $B$
(so $B$ is embedded into an idempotent complete category; it is
triangulated if $B$ is).

For a category $C$ we denote by $C^{op}$ its opposite category.

$\qv$ is the category of $\q$-vector spaces.

For a category $C$ and 
a 
partially ordered 
 index set $I$
 we will call a 
 set $X_i\subset \obj C$ a (filtered)
{\it projective system}
 if for any $i,j\in I$ there exists some maximum, i.e., an $l\in I$
 such that $l\ge i$ and $l\ge j$, and for all $j\ge i$ in $I$ there are fixed morphisms $X_j\to X_i$ satisfying the natural functoriality property. For such a system we have the natural notion of an inverse limit. Dually, we will call the inverse limit of a system of $X_i\in \obj C^{op}$ the direct limit of $X_i$ in $C$.
 
 All limits  
 (and pro-objects) in this paper will be filtered ones.

In this paper all complexes will be cohomological, i.e., the degree of
all differentials is $+1$; respectively, we will use cohomological
notation for their terms. $K^b(B)$ will denote the homotopy category of bounded complexes over an additive category $B$.

$M\in \obj C$ will be called compact if the functor $C(M,-)$
commutes with all small coproducts that exist in $C$ 
(we will only consider compact objects in those categories that are closed with respect to arbitrary small coproducts).

Dually, an object $M$ of $C$ is called cocompact if it is compact as an object of $C^{op}$.

$\cu$ and $\du$ will usually denote some triangulated categories.
 We will use the
term "exact functor" for a functor of triangulated categories (i.e.,
for a  functor that preserves the structures of triangulated
categories).

A class $D\subset \obj \cu$ will be called {\it extension-closed} if $0\in D$ and for any
distinguished triangle $A\to B\to C$  in $\cu$ we have the following implication: $A,C\in
D\implies B\in D$. In particular, an extension-closed $D$ is strict (i.e., contains all
objects of $\cu$ isomorphic to its elements).

The smallest extension-closed $D$ containing a given $D'\subset \obj \cu$ will be called the {\it extension-closure} of $D'$.

$D\subset \obj \cu$ will be called {\it suspended} (resp. {\it cosuspended}) if $D[1]\subset D$ (resp. $D\subset D[1]$). 

Below $\au$ will always  denote some abelian category.
We will 
usually assume that $\au$ satisfies AB5, i.e.,  it is closed with respect to all
small coproducts, and  filtered direct limits of exact sequences in
$\au$ are exact.

We will call a contravariant additive functor $\cu\to \au$
for an abelian $\au$ {\it cohomological} if it converts distinguished
triangles into long exact sequences. For a cohomological $H$ we will denote $H\circ [-i]$ by $H^i$.

 A functor $H:\cu\to \au$ below will always 
be cohomological; 
yet we will often use the letter $H$ also  for the object of a triangulated category (of the type $\dm(-)$) 
that  represents this functor.


For $X,Y\in \obj \cu$ we will write $X\perp Y$ if $\cu(X,Y)=\ns$.
For $D,E\subset \obj \cu$ we  write $D\perp E$ if $X\perp Y$
 for all $X\in D,\ Y\in E$.
For $D\subset \cu$ we  denote by $D^\perp$ the class
$$\{Y\in \obj \cu:\ X\perp Y\ \forall X\in D\}.$$
Sometimes we will denote by $D^\perp$ the corresponding
 full subcategory of $\cu$. Dually, ${}^\perp{}D$ is the class
$\{Y\in \obj \cu:\ Y\perp X\ \forall X\in D\}$. 

We will say that  $C_i\in \obj \cu$ {\it generate} $\cu$ if $\cu$ equals the Karoubi-closure of $\lan
C_i\ra$ in $\cu$. For a $\cu$  closed with respect to all small coproducts and $\du\subset \cu$ ($\du$ could be equal to $\cu$)
  we will say that 
$C_i$ generate $\du$ {\it as a localizing subcategory} if  
$\du$ is the smallest full strict triangulated subcategory of $\cu$ that contains $C_i$ and is closed with respect to all small coproducts.
 

All the 
  motivic categories of this paper 
  will be $\q$-linear ones. So, we can assume that the target categories of our cohomology theories (usually of the type $H:\dmcs\to \au$) are $\q$-linear ones.

  All morphisms and  schemes below will be separated. Besides, all schemes will be excellent noetherian of finite Krull dimension. 
$S$ will usually be our 
base scheme; $X$ will be of finite type over it. Often $j:U\to X$ will be an open immersion, and $i:Z\to X$ will be the complementary closed embedding.
 
 We will mostly be interested in finite type (separated) morphisms of schemes. In particular, we will say that a morphism is smooth  only if it is also of finite type.

For any scheme $Y$ we  denote by $Y_{red}$  the reduced scheme  associated with $Y$. 
It will always be possible in our arguments to replace $Y$ by $Y_{red}$.

 We will say that a projective system $X_i$, $i\in I$,  of schemes is {\it essentially affine} if the transition morphisms $X_j\to X_i$ are affine 
whenever $i\ge i_0$ (for some $i_0\in I$). Below we will only be interested in projective systems of this sort.

We will say that a scheme $X_0/S$ is essentially of finite type over $S$ if it can be presented as the (inverse) limit of finite type schemes $X_i/S$ such that all the transition morphisms between $X_i$ are open dense affine embeddings. We will say that $X_0$ is  (the spectrum of) a {\it 
   field over $S$} if $X_0$ is the spectrum of a field and
   is essentially of finite type over $S$.

	Below we will always assume that our "base scheme" $S$ is {\it reasonable} in the sense of (Definition 1.1 of) \cite{brelmot}. So we fix a certain  (separated noetherian finite-dimensional) "absolute" base scheme $S_0$ that we demand to be regular and of Krull dimension $\le 2$, and assume that $S$ if of finite type over $S_0$(. Moreover, it would be convenient for us to assume that all the schemes we consider in this paper are essentially of finite type over $S_0$; all the morphisms are $S_0$-ones. 

Below we will need two (possibly) somewhat distinct dimension functions for schemes; $\dim(-)$ will denote the ordinary Krull dimension. Respectively, "codimension" will mean the ordinary Krull codimension; thus we will say that $U\subset X$ is everywhere of codimension $c$ in it if the dimension of each connected component of $U$ is lesser by $c$ that the dimension of the connected component of $X$ containing it. 
We will say that a morphism is of dimension $s$ if 
 all of its fibres are equidimensional of Krull dimension $s$.

  \subsection{On dimension functions for excellent schemes}\label{sdimf}
	
	Below we will need a dimension function for $S_0$-schemes whose properties are somewhat better than the ones of the Krull dimension. So we recall a method for constructing a certain "nice" dimension function $\dimz$ (for essentially  finite type $S_0$-schemes) and formulate some properties of this function. For the convenience of the reader, we note that $\dimz$ (if we define it via the method described below)
	coincides with the Krull dimension if $S_0$ is of finite type over a field or over $\spe \z$ (more generally, it suffices to assume that $S$ is a Jacobson scheme all of whose components are equicodimensional; see Proposition 10.6.1 of \cite{ega43}). 
	Sometimes will call $\dimz(X)$ the {\it $\delta$-dimension} of $X$.
	
	Now, for a general (regular excellent finite dimensional separated) $S_0$ we apply the method described in Chapter XIV of \cite{illgabb}. Since we are only interested in excellent schemes, the corresponding theory is simpler than in the more general case considered in ibid.
	
	We start with "restricting" Definition 2.1.10 of ibid. to the setting of excellent schemes (see Proposition 2.1.6 of ibid.).
	
	\begin{defi}\label{ddf}
	Let $X$ be an excellent scheme; denote the set of its Zariski points by $\xx$.
	
	We will call a function $\dimz:\xx\to \z$ a {\it dimension function} if is satisfies the following condition: for any two $x,y\in \xx$ such that $y$ belongs to the closure of $x$ and the dimension of the localization of this closure at $y$ is $1$ we have $\dimz(x)=\dimz(y)+1$.

\end{defi}
	
	The definition immediately yields that such any two dimension functions for $X$ differ by a "shift function" that is constant on the connected components of $X$. Conversely, the sum of a dimension function with a locally constant one yields another dimension function. 
	Certainly, all dimension functions are bounded (if we fix a finite dimensional $X$).

	A reader can easily note that the results below essentially do not depend on the choice of a dimension function.
	Yet it will be convenient for us to  assume that all the dimension functions we consider below are non-negative. 

	We also note (applying Proposition 2.2.2 of ibid.) that for any scheme $Y$ whose connected components are integral one can define a "canonical"  dimension function using the following ("reasonable") formula  (for $y$ running through all points of $Y$)
	\begin{equation}\label{edimf} \dimz(y)=\dim Y^{y}-\dim \mathcal{O}_{Y,y},\end{equation}
	where $\mathcal{O}_{Y,y}$ is the corresponding local ring, $Y^y$ is the connected component of $Y$ containing $y$ (and $\dim(-)$ denotes the Krull dimension). 
	
	In particular, 
	({edimf}) can be applied for $Y=S_0$ (since $S_0$ is regular).
	
	Next we recall that for a fixed "base scheme" $Y$ and a fixed dimension function
	one can define "compatible" dimension functions for all schemes of finite type over $Y$.
	So,  any scheme $X$ that is essentially of finite type over $S_0$ possesses a dimension function  also; see Corollary 2.5.2 of ibid. Moreover, loc. cit. states that one can obtain a dimension function (that we will also denote by $\dimz$) for all finite type schemes over $S_0$ (simultaneously) as follows: for a field $X_0/S_0$ and $s_0$ being the image of $X_0$ in $S_0$ one should set $
	\dimz(X_0)=\dimz(s_0)+d$, where $d$ is the  transcendence degree of the (coordinate ring) extension $X_0/S_0$. 
	
	Next, for any essentially finite type $S_0$-scheme $X$ we define $\dimz(X)$ as the maximum of the corresponding "dimensions" of its points (that is easily seen to be finite and coincide with the maximum of "dimensions" of generic points of $X$).

	We will need a 
 'crude' notion of codimension (for essentially finite type $S_0$-schemes). 
For $T$ being a subscheme or a Zariski point of $ X$ 
we set $\codimz_X(T)=\dimz X-\dimz T$; we will call $\codimz_X(T)$ the {\it $\delta$-codimension} of $T$ in $X$. Obviously,  $\codimz_T(X)\ge 0$. We have an  equality  if $T$ is (open) dense in $X$
(or if it is the projective limit of a system of dense subschemes of $X$), and the inequality is strict if $X$ is connected and $T$ is closed in it. 
Note that $\codimz_T(X)$ coincides with the "usual" Krull codimension of $T$ in $X$ if $T$ and $X$ are irreducible; yet it often differs from the Krull codimension in the general case.
Respectively, it may make sense to call the ($\delta$)-coniveau spectral sequences introduced in Proposition \ref{pcss} below the niveau ones instead.

  \subsection{ $t$-structures: basics, compactly generated ones, and relation to gluing }\label{dtst}

In order to fix the notation, let us recall the definition of a $t$-structure.

\begin{defi}\label{dtstr}

A pair of subclasses  $\cu^{t\ge 0},\cu^{t\le 0}\subset\obj \cu$
for a triangulated category $\cu$ will be said to define a
$t$-structure $t$ if $(\cu^{t\ge 0},\cu^{t\le 0})$  satisfy the
following conditions:

(i) $\cu^{t\ge 0},\cu^{t\le 0}$ are strict (i.e., contain all
objects of $\cu$ isomorphic to their elements).

(ii) $\cu^{t\le
0}[1]\subset \cu^{t\le 0}$ (i.e., $\cu^{t\le 0}$ is suspended; see \S\ref{snotata}); $\cu^{t\ge 0}$ is cosuspended. 

(iii) {\bf Orthogonality}. $\cu^{t\le 0}[1]\perp
\cu^{t\ge 0}$.

(iv) {\bf $t$-decomposition}. For any $X\in\obj \cu$ there exists
a distinguished triangle
\begin{equation}\label{tdec}
A\to X\to B[-1]{\to} A[1]
\end{equation} such that $A\in \cu^{t\le 0}, B\in \cu^{t\ge 0}$.

\end{defi}

We will need some more notation. 

\begin{defi} \label{dt2}

1. A category $\hrt$ whose objects are $\cu^{t=0}=\cu^{t\ge 0}\cap
\cu^{t\le 0}$, $\hrt(X,Y)=\cu(X,Y)$ for $X,Y\in \cu^{t=0}$,
 will be called the {\it heart} of
$t$. Recall (cf. Theorem 1.3.6 of \cite{bbd}) that $\hrt$ is abelian
(short exact sequences in $\hrt$ come from distinguished triangles in $\cu$).

2. $\cu^{t\ge l}$ (resp. $\cu^{t\le l}$) will denote $\cu^{t\ge
0}[-l]$ (resp. $\cu^{t\le 0}[-l]$).

3. Assume  that triangulated categories $\cu$ and $\du$ are endowed with  certain $t$-structures $t_{\cu}$ and $t_{\du}$, respectively.
Then we will say that an exact $F:\cu\to \du$ is 
$t$-exact if $F(\cu^{t_{\cu}\ge 0})\subset \du^{t_{\du}\ge 0}$ and $F(\cu^{t_{\cu}\le 0})\subset \du^{t_{\du}\le 0}$).

\end{defi}

\begin{rema}\label{rts}

1. Recall (cf. Lemma IV.4.5 in \cite{gelman}) that (\ref{tdec})
defines additive functors $\cu\to \cu^{t\le 0}:X\to A$ and $C\to
\cu^{t\ge 0}:X\to B$. We will denote $A,B$ by $X^{t\le 0}$ and
$X^{t\ge 1}$, respectively.

The triangle (\ref{tdec}) will be called the {\it t-decomposition} of $X$. If
$X=Y[i]$ for some $Y\in \obj\cu$, $i\in \z$, then we  denote $A$
by $Y^{t\le i}$ (it belongs to $\cu^{t\le 0}$) and $B$ by $Y^{t\ge
i+1}$ (it belongs to  $\cu^{t\ge 0}$), respectively. 
Objects of the type $Y^{t^\le i}[j]$ and
$Y^{t^\ge i}[j]$ (for $i,j\in \z$) will be called {\it
$t$-truncations of $Y$}.

2. We denote by $X^{t=i}$ the $i$th cohomology of $X$ with respect
to $t$, i.e., the object $(X^{t\le i})^{t\ge 0}$ (cf. part 10 of \S IV.4 of
\cite{gelman}). 

3.  The following statements are obvious (and well-known): $\cu^{t\le 0}={}^\perp
\cu^{t\ge 1}$; $\cu^{t\ge 0}= \cu^{t\le -1}{}^\perp$.

4. We will say that $t$ is non-degenerate if $\cap_{l\in \z}\cu^{t\ge l}=\cap_{l\in \z}\cu^{t\le l}=\ns$.

It is easily seen that 
the non-degeneracy $t$  is equivalent to the following: $M\in\obj \cu$ is zero whenever $M^{t=i}=0$ for all $i\in \z$.  


\end{rema}

Now let us recall an important result on the existence of $t$-structures. It will be convenient for us to use the following definition.

\begin{defi}\label{daisle}
Let $\cu$ be a triangulated category closed with respect to all small coproducts.

1. We  call a suspended class $D\subset\obj \cu$ a {\it preaisle} if it is closed with respect to all coproducts and is extension-closed.

2. For a $E\subset \obj \cu$ we will call the smallest preaisle containing $E$ the {\it preaisle generated by $E$}. 

\end{defi}  
  
  It turns out that "compactly generated" preaisles are {\it aisles} in the terms of \cite{talosa}.

  \begin{pr}\label{paisle}
Assume that $\cu$ is  closed with respect to  coproducts; let $E\subset \cu$ be a set of compact objects. Then the following statements are valid.

1. There exists a unique $t$-structure $t$ for $\cu$ such that $\cu^{t\le 0}$ is the preaisle generated by $E$. 

2. The corresponding truncation functors $-^{t\le 0}$,  $-^{t\ge 0}$, and $-^{t=0}$ (see Remark \ref{rts}(1,2)) respect coproducts.

3. The corresponding $\cu^{t\ge 0}$ equals $(\cup_{i> 0}E[i])^\perp$. 

\end{pr}
\begin{proof} 1. This is Theorem A.1. of \cite{talosa}.

2. Immediate from Proposition A.2 of ibid. 

3. This is an easy consequence of Remark \ref{rts}(3). 


\end{proof}

  Lastly, let us recall the notion of gluing for triangulated categories and $t$-structures. 
  
  \begin{defi}\label{dglu}
1. The $9$-tuple $(\cu,\du,\eu, i_*, j^*, i^*, i^!, j_!, j_*)$  is called a
{\it gluing datum} if it satisfies the following conditions.

(i) $\cu,\du,\eu$ are triangulated categories; $i_*:\du\to\cu$,
$j^*:\cu\to\eu$, $i^*:\cu\to\du$, $i^!:\cu\to\du$, $j_*:\eu\to\cu$,
$j_!:\eu\to\cu$ are exact functors.

(ii) $i^*$ (resp. $i^!$) is left (resp. right) adjoint to $i_*$;
$j_!$ (resp. $j_*$) is left (resp. right) adjoint to $j^*$.

(iii)  $i_*$ is a full embedding; $j^*$ is isomorphic to the
localization (functor) of $\cu$ by
$i_*(\du)$.

(iv) For any $M\in \obj \cu$ the pairs of morphisms $j_!j^*M \to M
\to i_*i^*M$ and $i_*i^!M \to M \to j_*j^*M$ can be completed to
distinguished triangles (here the connecting
morphisms come from the adjunctions of (ii)). Moreover, these completions are functorial in $M$.


(v) $i^*j_!=0$; $i^!j_*=0$.

(vi) All of the adjunction transformations $i^*i_*\to \id_{\du}\to
      i^!i_*$ and $j^*j_*\to \id_{\eu}\to j^*j_!$ are isomorphisms of
      functors.

2. In the setting of part 1, assume also that $\du$ and $\eu$ are endowed with  certain $t$-structures $t_{\du}$ and $t_{\eu}$, respectively.
 Then we will say that a $t$-structure $t=t_{\cu}$ for $\cu$ is {\it glued from}  $t_{\du}$ and $t_{\eu}$ if 
we have the following: $ \cu^{t_{\cu}\le 0}=\{M\in \obj \cu:  j^*M\in \eu^{t_{\eu}\le 0}\text{ and }  
i^*M\in \du^{t_{\du}\le 0}\}$, and $ \cu^{t_{\cu}\ge 0}=\{M\in \obj \cu:  j^*M\in \eu^{t_{\eu}\ge 0}\text{ and }  
i^!M\in \du^{t_{\du}\ge 0}\}$.  In this case we will also say that $\cu$, $\du$, and $\eu$ are {\it endowed with compatible $t$-structures}.


\end{defi}

\begin{rema}\label{rglu}
1. Our definition of a gluing datum is far from being the 'minimal' one. Actually, it is 
 well known (see Chapter 9 of \cite{neebook}) that
a gluing datum can be uniquely recovered from an inclusion
$\du\to \cu$ of triangulated categories that admits both a left
and a right adjoint functor.

Our notation for the connecting functors 
is (certainly) coherent with Theorem \ref{tcisdeg} 
 below. 

2. The following fact is well-known (see Proposition 1.5.3(II) of \cite{bmm} and Proposition 2.3.3 of \cite{degcis}): in the gluing setting of Definition
 \ref{dglu}(1) $t_{\cu}$  is glued from  $t_{\du}$ and $t_{\eu}$
if and only if $i_*$ and $j^*$ are $t$-exact.
\end{rema}

  \subsection{On motives over a base: reminder}\label{sdm}

We recall some basic properties of the triangulated categories of Beilinson motives described by Cisinski and D\'eglise. 
Recall that all the schemes we mention are reasonable (see \S\ref{snotata}), and all the morphisms considered are separated. 

In the following statement we assume that
$f:X\to Y$ is a  
finite type morphism of schemes; $S$ is a finite type $S_0$-scheme.

\begin{theo}\label{tcisdeg}
 The following statement are valid.

\begin{enumerate}

\item\label{imotcat} For any (excellent separated finite dimensional) $X$ a tensor triangulated $\q$-linear category $\dmx$ with the unit object $\q_X$ is defined; it is closed with respect to arbitrary small coproducts and zero if $X$ is empty.

$\dmx$ is the category of  {\it Beilinson motives} over $X$, as described (and thoroughly studied) in \S14--15 of \cite{degcis}.

\item\label{imotgen}
The (full) subcategory $\dmcx\subset \dmx$ of compact objects is tensor triangulated, and  $\q_X\in \obj \dmcs$. $\dmcx$ 
generates $\dmx$ as (its own) localizing subcategory. 

\item\label{ivoemot} If $X$ is the spectrum of a perfect field $F$, $\dmcx$ is canonically isomorphic to the category $\dmgm=\dmgm(F)$ of Voevodsky's geometric motives (with rational coefficients) over $X$ (see \cite{1}). Moreover, the whole $\dmx$ is isomorphic to the category $\dm(F)$ considered in \cite{degmod}.


\item\label{imotfun}  
The following functors
 are defined:
$f^*: \dm(Y) \leftrightarrows \dmx:f_*$ and $f_!: \dmx \leftrightarrows \dmy:f^!$; $f^*$ is left adjoint to $f_*$ and $f_!$ is left adjoint to $f^!$.

We call these the {\bf motivic image functors}.
Any of them (when $f$ varies) yields a  $2$-functor from the category of 
(separated finite-dimensional excellent) schemes
with  morphisms of finite type to the $2$-category of triangulated categories.
Besides,
all motivic image functors preserve compact objects (i.e. they could be restricted to the subcategories $\dmc(-)$); they also commute with arbitrary (small) coproducts. 

\item\label{iexch} 
For a Cartesian square
of finite type 
morphisms
$$\begin{CD}
Y'@>{f'}>>X'\\
@VV{g'}V@VV{g}V \\
Y@>{f}>>X
\end{CD}$$
of schemes we have $g^*f_!\cong f'_!g'{}^*$ and $g'_*f'{}^!\cong f^!g_*$.

\item\label{itate}  
There exists a
Tate object $\q(1)\in\obj\dmcx$; tensoring by it yields an  exact Tate twist functor $-(1)$ on $\dmx$.
This functor is an auto-equivalence of $\dmx$; we will denote the inverse functor by $-(-1)$.

 Tate twists
commute with all motivic image functors mentioned (up to an isomorphism of functors).

Besides, for $X=\p^1(Y)$ there is a functorial isomorphism $f_!(\q_{\p^1(Y)})\cong \q_Y\bigoplus \q_Y(-1)[-2]$. 

\item\label{iupstar}  $f^*$ is symmetric monoidal; $f^*(\q_Y)=\q_X$.

\item \label{ipur}

$f_*\cong f_!$ if $f$ is proper;
$f^!(-)\cong f^*(-)(s)[2s]$ 
 if $f$ is smooth  
  of  dimension $s$ (i.e. it is of relative dimension $s$ everywhere; see \S\ref{snotata}). 

If $f$ is an open immersion,
 we just have $f^!=f^*$.

\item \label{ifuh}
If $f$ is a finite universal homeomorphism then the functors $f^*=f^!$ and $f_*=f_!$  are equivalences of categories.
In particular, this is the case where $X=Y_{red}$ (the reduced scheme associated with $Y$).
 
Moreover, $g^*$ is an equivalence of categories also in the case where $g$ is an inverse limit of finite universal homeomorphisms. 

\item \label{ivanmc}

Assume that $X$ (or  $X_{red}$ as defined in \S\ref{snotata})) is regular of Krull dimension $d$ (everywhere). Then for any $i,j\in \z$ such that $d-i<-j$ we have 
$\dm(X)(\q_X,\q_X(i)[j])=\ns$.

\item \label{ipura} Assume that 
 $f$ is an immersion of schemes everywhere of (Krull) codimension $c$, whereas $Y$ and  $X_{red}$  are regular.
 Then $\q_{X}(-c)[-2c]\cong f^!(\q_Y)$.

\item\label{iglu}

If $i:Z\to X$ is a closed embedding, $U=X\setminus Z$, $j:U\to X$ is the complementary open immersion, then
the motivic image functors yield a  gluing datum for $\dm(-)$  (in the sense of 
Definition \ref{dglu}; one should set $\cu=\dmx$, $\du=\dm(Z)$, and $\eu=\dmu$ in it).


\item\label{ridmot} The functor $g^*$ on $\dmc(A)\subset \dm(A)$
 can be defined for any (separated) morphism $g: B\to A$ not necessarily of finite type; we have $g^*\q_A=\q_B$.  This definition respects the composition for morphisms and  satisfies the first part of assertion \ref{iexch}. Moreover, $g^*$ "respects $-_*$" if $g$ is a limit of an essentially affine system of 
open (dense) embeddings. 


\item\label{icont}
Assume that a scheme  $X$  is the limit of an essentially affine (filtering) projective system of  schemes $X_\be$ (for $\be\in B$). Then $\dmcx$ is isomorphic to the $2$-colimit 
 of the categories $\dmc(X_\be)$; in these isomorphisms all the connecting functors are given by the corresponding $(-)^*$ (cf. the previous assertion).

\item\label{imotalt} For any finite type $Z/S$  there exists an open dense $U\subset Z$ and a finite morphism $U'\to U$ such that for the composite morphisms 
$u:U\to S$ and $u':U\to S$ the following is valid:
$u_!\q_{U}$ is a retract of
 $u'_!\q_{U'}$ and $U'$ is open in some regular scheme that is projective over $S$.

 \item\label{igenc}
 $\dmcs$ 
  is generated 
   by $\{ x_*(\q_X)(r)\}$, where $x:X\to S$ runs through all projective morphisms  (of finite type) such that $X$ is regular, $r\in \z$.  

\end{enumerate}

\end{theo}
\begin{proof}

Most of these statements were stated in the introduction of \cite{degcis} (and proved later in ibid.); see \S1.1 of 
 \cite{brelmot} for more details. Besides, the fact that  $g^*$ "commutes with $-_*$" if $g$ is a limit of (an essentially affine system of) open embeddings (see assertion \ref{ridmot})
 is a particular case of Proposition  4.3.14 of \cite{degcis}.

 
 Assertion \ref{ifuh} for finite type universal homeomorphisms follows from Proposition 2.1.9  of \cite{degcis} (note that we can apply the result cited by 
Theorem 14.3.3 of ibid.) via applying the 
 corresponding adjunctions.  The second half of the assertion is an easy consequence of Proposition 2.2.1(14) of \cite{bmm}.
 
 In assertion \ref{ivanmc} (as well as in the following one) one may replace $X$ with $X_{red}$ according to assertion \ref{ifuh}. 
 In this case Corollary 14.2.4 of ibid. yields that the group in question is isomorphic to $Gr_{\gamma}^iK_{2i-j}(X)$. The vanishing of this motivic cohomology group for a regular $X$ of Krull dimension $\le d$
 is well-known; see (the  "$j<-n$ part" of) Theorem 8(1) of \cite{souoper}.  

Assertion \ref{imotalt} 
can be easily obtained via the argument used in the proof of Proposition 15.2.3 of \cite{degcis} (see also \S4.4 of ibid.). 

\end{proof}

  \section{On Borel-Moore motives of $S$-schemes and coniveau spectral sequences}\label{sbm}

  In \S\ref{smgeff} we  introduce 
certain Borel-Moore motives for 
 finite type $S$-schemes and 
  discuss their properties; the most important of them is Theorem \ref{tmgeff}(II.1). We introduce some more tools for studying their functoriality in \S\ref{sbmp}.
  
  In \S\ref{scss} we study in detail the coniveau spectral sequences for the cohomology of (the Borel-Moore motives of) $S$-schemes (note that in the proof of Theorem \ref{thts}(I) below only part I.1 of  Proposition \ref{pcss} is applied).

\subsection{On Borel-Moore motives for 
$S$-schemes: the main properties}\label{smgeff}

We 
introduce 
certain Borel-Moore motives for (arbitrary) finite type $S$-schemes (and recall some notions introduced in previous papers on Beilinson motives).

\begin{rema}\label{rmgbm}

\begin{enumerate}
\item\label{img}

 For a 
finite type  $x:X\to S$ we set $\mgbms(X)=x_!\q_X$ (this is a certain {\it Borel-Moore} motif of $X$; cf. Remark \ref{rnat}(2) below). 

These objects will be very important for us. We will study their "functoriality" in detail below. In the proof of the next theorem  
it suffices to apply the adjunctions of the type  $j_!\dashv j^*$ 
for an open embedding $j:U\to X$. 

\item\label{ino} Yet (in order to introduce some notation) we recall now that in \cite{degcis} for any smooth morphism $f:X\to Y$ the functor $f_\sharp:\dmcx\to \dmcy$ (and also $\dm(X)\to \dm(Y)$) left adjoint to $f^*$  was considered (and played an important role; we will say more on this below). Now, Theorem \ref{tcisdeg}(\ref{imotfun},\ref{ipur}) yields that for a smooth $f$ everywhere of dimension $s$ we have an isomorphism $f_\sharp \to f_!(-)\lan s\ra$. Moreover, if $f$ is an open embedding
(and so, $s=0$) then $f_!=f_\sharp$ by construction.

So, for a smooth morphism $g$ 
 we will denote by $M_g$ the counit of the adjunction $g_\sharp\dashv g^*$.
Note that for smooth morphisms $Y'\stackrel{g'}{\to} Y\stackrel{g}{\to} X$, $O\in \obj \dm(X)$ we have $M_{g\circ g'}(O)=M_g(O)\circ g_\sharp(M_{j'}(j^*O))$.

Besides, for any (separated) $g$ we will denote by $M^g$ the unit of the adjunction $g^*\dashv g_*$.


\item\label{itw} It will be convenient for us to use the following notation for "shifted Tate twists": for any $r\in \z$ we  denote the functor $-(r)[2r]$ (resp. $-(r)[r]$) by $-\lan r\ra$ (resp. by $-\{r\}$; both of these functors will play important roles).

 \item\label{ichows} The following definition was central in \cite{hebpo} and \cite{brelmot}: $\chows$ is the Karoubi-closure of $\{x_!(\q_X)\lan r\ra\}= 
\{x_*(\q_X)\lan r\ra\}$ in $\dmcs$; here $x:X\to S$ runs through all finite type projective  morphisms such that $X$ is regular, $r\in \z$.

We define $\chowes\subset \chows$ as the Karoubi-closure in $\dmcs$ of the class  $\mgbm(X)\lan \dimz(X)\ra$ for $X$ running through (regular) projective 
 $S$-schemes (note that this definition depends on $\delta$ in the obvious way).

One can easily check that one obtains 'the usual' Chow motives and effective Chow motives this way in the case where $S$ is the spectrum of a perfect field (and $\delta$ is the Krull dimension of schemes; cf. Theorem \ref{tcisdeg}(\ref{ivoemot})). Moreover, if $S$ is a spectrum of an imperfect field, one obtains motives over its perfect closure; cf. Theorem \ref{tcisdeg}(\ref{ifuh}).

\end{enumerate}

\end{rema}


Now let us establish 
the main properties of $\mgbm(-)$.

\begin{theo}\label{tmgeff}

Denote by $\dms^{\le 0}$ the preaisle generated by $\{
\chowe(S)\{i\},\ i\in \z\}$ (see Definition \ref{daisle}).
 Let $x:X\to S$ be a finite type morphism;  denote  $\dimz(X)$ by $d$. 

I The following statements are valid. 

\begin{enumerate}
\item\label{igys}
Let $i:Z\subset X$ be a closed embedding; 
 denote by $j:U\to X$ the complementary open embedding. 
 Then (in the notation introduced in the previous remark) there exists a distinguished triangle 
 \begin{equation}\label{emgys}
\mgbms(U)\stackrel{x_!(M_j(\q_X))} 
{\to}\mgbms(X) 
{\to} \mgbms(Z). 
\end{equation}
Moreover, if $Z$ and $X$ are regular, $Z$ is everywhere of codimension $c$ in $X$, then there is also a distinguished triangle
\begin{equation}\label{emgeffc} 
z_* \q_Z\lan -c\ra \to x_* \q_X\to u_*\q_U
\end{equation} 

\item\label{idecomp} $\dms^{\le 0}$ contains $\mgbms(X)\lan d\ra$. 

\end{enumerate}

II 
Assume that $X$ is irreducible;  consider the generic point  $X_0$ of $X$ as the (inverse) limit of open affine subschemes $X_i$ of $X$ (where $i$ runs through a filtering projective system $I$). For  an 
$H\in \obj \dms$ we define $H(X_0)$ as  $\inli \dmcs(\mgbms(X_i),H)$ 
(where for $j_{i_1i_2}$ being  the connecting maps $X_{i_1}\to X_{i_2}$ the corresponding morphisms in the direct limit 
are induced by $M_{j_{i_1i_2}}$; see
Remark \ref{rmgbm}(\ref{ino})).

Then $H(X_0)=\ns$ in the following cases.

1. $H=y_*\q_Y(s-d+e)[s-2d+2e+r]$ for some 
finite type $y:Y\to S$, where $Y$ is regular 
 connected, $e=\dimz(Y)$, $r>0, s\in \z$. 

2. $H\lan d \ra [-1]$ belongs to  $\dms^{\le 0}$.  

\end{theo}
\begin{proof}
I.1 
We can assume that $X$ is connected. Theorem \ref{tcisdeg}(\ref{imotfun},\ref{iglu},\ref{ipur}) yields  distinguished triangles $j_!j^*\q_X(\cong j_!\q_U)\stackrel{M_j(\q_X)}{\to} \q_X \stackrel{M^i(\q_X)}{\to} i_!\q_Z\cong i_*i^*\q_X $ and $i_!i^!\q_X \to 
\q_X\to  j_*j^*\q_X(\cong 
 j_*\q_U) $.
Applying the functor $x_!$ to the first triangle 
we obtain (\ref{emgys}). 

Next, if $Z$ and $X$ are regular, then $i_!i^!\q_X\cong i_*\q_Z\lan -c \ra$ (see Theorem \ref{tcisdeg}(\ref{ipura})). Hence the  application of $x_*$ to our second distinguished triangle 
yields (\ref{emgeffc}).

2. 
 We prove the statement by induction on $d$. 

For our $X$ choose an open dense $U\subset X$ as in Theorem \ref{tcisdeg}(\ref{imotalt}). 
 Since $\dimz (X\setminus U)< d$, 
  we can assume that $\mgbms(X \setminus U)\lan d\ra \in \dms^{\le 0}$. 
Hence (\ref{emgys}) 
yields that $\mgbm(X)\lan d\ra\in \dms^{\le 0}$ if and only if $\mgbm(U)\lan d\ra\in \dms^{\le 0}$. 

So, we replace $X$ by $U$. Theorem \ref{tcisdeg}(\ref{imotalt}) yields that it suffices to consider the case where $X$ is open dense in some regular $X'$ that is projective over $S$. Applying 
"the birationality of $\mgbm(X)\lan d\ra\in \dms^{\le 0}$" again we obtain the result. 


II.1. Denote the connecting morphisms $X_i\to X$ by $j_i$ for all $i\in I\cup\ns$. 
By Theorem \ref{tcisdeg}(\ref{imotfun},\ref{ipur},\ref{itate}) 
the group in question is isomorphic to $$\begin{gathered} \dmc(X)(j_{i!}\q_{X_i},x^!y_*\q_Y(s-d+e)[s-2d+2e+r])\\  \cong \inli_{i\in I}\dmc(X_i)(\q_{X_i},j_i^*x^!y_* 
\q_Y(s-d+e)[s-2d+2e+r]) .\end{gathered} $$ 
 Applying Theorem \ref{tcisdeg}(\ref{icont},\ref{ridmot}) we transform this into
$\dmc(X_0)(\q_{X_0},j_0^*x^!y_*\q_Y(s-d+e)[s-2d+2e+r]) $. 

Now, let us make certain reduction steps.

Certainly, we can assume that $X$ is quasi-projective over $S$. Consider a factorization of $x$ as $X\stackrel{f}{\to}S'\stackrel{h}{\to}S$ where
 $h$ is smooth of dimension $q$,
 $f$ is an 
embedding, $S'$ is connected, and consider the corresponding diagram
$$ \begin{CD} 
 Z@>{f_Y}>> Y'@>{h_Y}>> Y\\
 @VV{z_X}V
@VV{y'}V@VV{y}V \\
    X@>{f}>> S' @>{h}>> S
 \end{CD}$$
 (the upper row is the base change of the lower one to $Y$).
 Then we have $x^!y_*\q_Y\lan e\ra= f^!h^!y_*\q_Y\lan e\ra\cong f^! y'_*h_Y^!\q_Y\lan e\ra$
 (by Theorem \ref{tcisdeg}(\ref{iexch})). Parts \ref{ipur} and \ref{iupstar} of the theorem allow to transform this into 
 $f^! y'_*\q_{Y'}\lan e+q\ra$. Hence below we may assume  $x$ is an embedding (since we can replace $S$ by $S'$ in our assertion; note that $e+q=\dimz(Y')$). Besides, the isomorphism $x^!y_*\cong z_{X*}z_Y^*$ for $z_Y=h_Y\circ f_Y$ yields that the group in question is zero if $Y$ lies over 
 $S\setminus X$ (considered as a set); see Theorem \ref{tcisdeg}(\ref{imotcat}).
 
 Next, we  note that  it suffices to verify the statement for $Y$ replaced by the components of some its regular   connected stratification. 
Being more precise,  
  we present $Y$ as $\cup_{1\le i \le n} Y_l$,
 where $n\ge 0$, $Y_l$ ($1\le l\le n$) 
 are pairwise disjoint 
 subschemes of $S$, and each $Y_l$ is open in $\cup_{i\ge l} Y_i$; we assume that all $Y_l$ are regular and connected. 
Indeed, 
 for such a stratification (\ref{emgeffc}) implies that  $y_*\q_Y\lan e\ra$ belongs to the extension-closure of $y_*j_{l*}\q_{Y_l}\lan \dimz(Y_l)\ra$; 
 here $j_l:Y_l\to Y$  are the corresponding embeddings. 

Now, we can choose a stratification  
 of this sort such that each of $Y_l$  lies either over $X$ or over $S\setminus X$. 
  Therefore it suffices to verify our assertion in the case where $y$ factors through $x$. Moreover, since $x^!x_*$ is the identity functor on $\dmc(X)$ (in this case; see Theorem \ref{tcisdeg}(\ref{iglu})), we may also assume that $X=S$. 
 
 Consider the following Cartesian square: $$ \begin{CD} 
 Y_0@>{j_{0Y}}>>  Y\\
 @VV{y_{0}}V
@VV{y}V \\
    X_0@>{j_0}>> S
 \end{CD}$$
  We have $j_0^*y_*\q_Y 
  \cong  y_{0*}j_{0Y}^*\q_Y 
  = y_{0*}\q_{Y_0}$ 
  (see Theorem \ref{tcisdeg}({ridmot})). Hence the adjunction 
  $  y_{0}^*\dashv y_{0*}$ 
	yields that the group in question is isomorphic to $\dmc(Y_0)(\q_{Y_0},\q_{Y_0}(s-d+e)[s-2d+2e+r])$. Since the Krull dimension of $Y_0$ is $e-d$ (unless $Y_0$ is empty; see \S\ref{sdimf}),  Theorem \ref{tcisdeg}(\ref{ivanmc}) 
  yields our assertion.

2. 
 Since $\mgbms(Y)$ is compact in $\dms$ for any finite type $Y/S$, 
 assertion I.2  easily yields the following (see  Corollary \ref{cres}(I) below): 
 it suffices to consider the case where $H\lan d \ra=\mgbms(Y)\lan \dimz(Y)\ra \{s\}[r]$ for some $s\in \z$, $r> 0$, and a regular  connective $Y$ that is projective over $S$. 
Hence the statement follows from the previous assertion.


\end{proof}

\subsection{On "smooth models" for Borel-Moore motives of subschemes}\label{sbmp}

The main difficulty of dealing with Borel-Moore motives as described above is that they are not "strictly functorial" with respect to motivic pullbacks. One of the ways to overcome this difficulty (proposed by F. Deglise) 
 is to study the  functoriality of the first isomorphism in Theorem \ref{tcisdeg}({iexch}) and of the gluing distinguished triangles.
In the current method we use another method; so we introduce certain alternative descriptions for 
$\mgbmx(T)$ for a  subscheme $T$ of $X$. These objects will be connected with $\mgbmx(T)$ by certain canonical isomorphisms. 
Thus in the 
next subsection we could have applied Theorem \ref{tmgeff}(I.\ref{igys}) instead of 
Proposition \ref{pmg} below; this substitute would just have  made the corresponding morphisms between   $\delta$-coniveau spectral sequences 
non-canonical. Note also that the corresponding (weakened) version of the proposition is quite  sufficient for the proof of Theorem \ref{thts} 
(I) (below).

Let us start with a certain reminder and remarks.

\begin{rema}\label{rsharp}

1. For a smooth $x:X\to S$ in ibid. the object $\mgs(X)=f_\sharp \q_X$ (see Remark \ref{rmgbm}(\ref{ino})) was considered in \cite{degcis}. The idea is that the cohomology of  $\mgs(X)$ with coefficients in an $H\in \obj \dms$ is isomorphic to the cohomology of $\q_X$ with coefficients in $f^*H$, whereas for $H$ representing (for example) motivic or \'etale cohomology over $S$ the object $f^*H$ represents the corresponding cohomology over $X$. So, $\mgs(X)$ "yields the cohomology of $X$".

Now we note that for   a smooth and equidimensional $x$ the motif  $\mgs(X)$ differs from $\mgbms(X)$ only by a twist (and it could have made some sense to modify the notation of Definition \ref{dcoh}(\ref{i2}) by this twist). Besides, we also have a similar relation for the $H$-cohomology in the case where $X,S$ are smooth (and equidimensional) over some common base $S'$ and $H$ "comes from $S'$ via $f^*$" (for $f:S\to S'$). 
We also have a similar relation of the $\delta$-coniveau spectral sequences for $H^*(\mgbm(X))$ that we will study below (see Proposition \ref{pcss}) to the ones for the 'usual' $H$-cohomology of $X$. In particular, one can apply this observation to the study of \'etale and motivic cohomology (since those are defined over $\spe \z$). 
So, studying the cohomology of $\mgbm(-)$ (as we do in the current paper) and its functoriality is quite actual. 

Moreover, Borel-Moore motives are related to the 
{\it Borel-Moore motivic homology} and $K'$-theory; see \S I.IV.2.4.10-5.3 of \cite{lemm} and Remark \ref{rnat}(2) below.

2. By Remark 3.3.6 of \cite{degcis},  we have the Nisnevich Mayer-Vietoris triangles for the motives $\mgs(-)$. 
In particular, we have $\mgs(X\sqcup Y)=\mgs(X)\bigoplus \mgs(Y)$  for any 
smooth $X,Y/S$.


\end{rema}

We will need a certain 
extension of the functor $\mgs$ 
mentioned.  Let us start with introducing some auxiliary notation. 

\begin{defi}\label{dsms}
1. $\sms$ is the category whose objects are smooth 
$S$-schemes and morphisms are formal linear combinations of their morphisms. 

Certainly, the direct sum operation for $\smx\subset K^b(\smx)$ comes from the disjoint union of smooth $X$-schemes, whereas $\emptyset$ yields a zero object of this category.

2. For any $g:X'\to X$  denote by $\sm_g:K^b(\smx)\to K^b(\smxp)$
the functor induced by $-\times_X X'$.

3. If $f:X\to Y$ is smooth we denote by $\sm^g: K^b(\smx)\to K^b(\smy)$ the natural functor of "restricting the base".

4. Below for $f:X\to Y$ being a morphism of smooth 
$S$-schemes we will denote by $\mgs(Y)$ (resp.  $\mgs(X\stackrel{f}{\to} Y)$; we will sometimes omit $f$ in this notation) the result of applying the functor $\mgs:K^b(\sms)\to \dmcs$ (see Proposition \ref{pmg}(I.\ref{ikbsmo}) below) to the $\sms$-complex $\dots \to 0\to Y\to 0\dots $ (resp.   $\dots \to 0\to X\stackrel{f}{\to} Y\to 0\dots $); here we put $Y$ in degree $0$.

5. For $T$ being a reduced 
closed subscheme of some smooth $U/S$
 we denote $\mgs(V\to U)$ by $\mgs^U(T)$, 
where $V=U\setminus T$.

\end{defi}


Now introduce our version of $\mgs$ and prove its main properties.

\begin{pr}\label{pmg}
I. Let $g:X'\to X$ be an arbitrary (separated) morphism of reasonable schemes.

Then the following statements are valid.

\begin{enumerate}

\item\label{ikbsmo} For any (reasonable) $X$ there is an exact functor $\mgx:K^b(\smx)\to \dmcx$. For a smooth 
 $w:W\to X$ we have $\mgx(W)=w_\sharp \q_W$. 

\item\label{ikbsmii} 
We have $g^*\circ \mgx=\mgxp\circ\sm_g$.

\item\label{ikbsme} If $f:X\to Y$ is smooth, 
 we have 
$f_\sharp \circ \mgx= \mgy\circ \sm^f$ (recall also that $f_\sharp=f_!$ if $f$ is an open embedding). Moreover, for any $C\in \obj K^b(\sm_Y)$ the morphism $M_f(\mg_Y(C))$ is the natural morphism $\mg_Y(\sm^f\circ \sm_f(C)\to C)$
 
\end{enumerate}

II Let $S,T,U,V$ be as in definition \ref{dsms}(5); let $h:S'\to S$ be a (separated) morphism; denote by $U'$ (resp. $T'$) the base change from $S$ to $S'$ of $U$ (resp. the reduced scheme associated with the corresponding base change of $  T$). 
Then the following statements are valid.

\begin{enumerate}

\item\label{iii} 
We have  $h^*\mgs^U(T)=\mg_{S'}^{U'}(T')$.

\item\label{iei} Assume that  $h$ is smooth. 
Then $h_\sharp\mg_{S'}^{U'}(T') = \mg_{S}^{U'}(T')$. 


\item\label{icho} For fixed $S,X,T$ all the objects of the type $\mgs^U(T)$ are related by canonical isomorphisms. In particular, if $T$ is open in $X$ then all of these
objects are isomorphic to $\mgs^T(T)$.

\item\label{inis} Assume that $T=T_1\sqcup T_2$. Then $\mgs^U(T)\cong \mgs^U(T_1)\bigoplus \mgs^U(T_2)$

\item\label{iop} Assume that $h$ is an open embedding,  $T\subset S'\times_S U$ (whence $T'= T$). Then the morphism $M_h(\mgs^U(T))$ (see Remark \ref{rmgbm}(\ref{ino})) is an isomorphism.

\item\label{icl} Assume that $h$ is a closed embedding and $T\subset S'\times_S U$ (again).  Then 
$M^h(\mgs^U(T))$ 
is an isomorphism.

\item\label{iso} 
If $T\subset S'\subset S$ then there is a canonical isomorphism 
$\mg_S^U(T)\to h_! \mg_{S'}^{U'}(T)$. 
 In particular, taking $S'=T$ we obtain $\mgbm(T)\cong \mg_S^U(T)$.

\end{enumerate}

\end{pr}
\begin{proof}

I 
Immediate from one of the descriptions of $\dmc(-)$; see 
Theorem 16.1.2 of \cite{degcis}.

II Assertions \ref{iii},\ref{iei} are just particular cases of assertions I.\ref{ikbsmii} and I.\ref{ikbsme}, respectively. 

\ref{icho}. It suffices to note that for $U_0$ being an open neighbourhood of $T$ in $U$ the natural morphism  $\mgs^{U_0}(T)\to  \mgs^U(T)$ is an isomorphism. The latter is an immediate consequence of Remark \ref{rsharp}(2). Loc. cit. also easily implies assertion \ref{inis}.

\ref{iop}. By assertions (\ref{iii},\ref{iei}), we have $h_!h^*\mgs^U(T)=\mgs^{U'}(T)$. Hence the statement is an immediate consequence of assertion \ref{icho}.

\ref{icl}. 
Denote the open embedding complementary to $h$ by $j$. Assertion \ref{iii} yields that $j^*\mgs^U(T)=0$.
Hence it suffices to note that $h^*$ yields an equivalence of categories when restricted to the "kernel" of $j^*$ (see Definition \ref{dglu}(1) and Theorem \ref{tcisdeg}(\ref{iglu})). 

\ref{iso}. Denote the closure of $S'$ in $S$ by $S''$; $h'':S''\to S$ is the corresponding closed embedding; let $h'$ be the (open) embedding of $S'$ into $S''$.

Assertion 
 \ref{iii} yields that $h''^*\mgs^U(T)=\mg_{S''}^{U\times_S S''}(T)$. 
Next, by assertion \ref{iei}  we have $h'_! \mg_{S'}^{U'}(T)=\mg_{S''}^{U'}(T)$. Hence it remains to note that (according to  the previous
assertions) the morphisms $M^{h''}:\mg_{S'}^{U'}(T)\to h''_*\mg_{S''}^{U\cap S''}(T)=h''_!\mg_{S''}^{U\cap S''}(T)$ and $h''_!(M_{h'}(\mg_{S''}^{U\cap S''}(T)):h''_!\mg_{S''}^{U\cap S''}(T)\to h_! \mg_{S'}^{U'}(T)$ are isomorphisms.


\end{proof}

\begin{rema}\label{rnat}
1. Since the isomorphism constructed in the proof of assertion II.\ref{iso} is obtained by composing a natural transformation with its inverse, it is compatible with the "change of $U$" isomorphisms mentioned in assertion II.\ref{icho}. This means (in the case where $U_0\subset U$, where $U_0$ is an open neighbourhood of $T$ in $U$) that the following square 
$$\begin{CD}
\mg_S^{U_0}(T)@>{}>>\mg_S^{U}(T)\\
@VV{}V@VV{}V \\
h_!\mg_{S'}^{U_0\cap S'}(T) @>{}>>h_!\mg_{S'}^{U'}(T)
\end{CD}$$
is commutative. 
Here we also apply assertion 
I.\ref{ikbsmii} 
in order to 
identify the bottom arrow in this square.

2. One can 'almost generalize' assertion II.\ref{iso} in the following way. Assume that $U$ is everywhere of dimension $s$ over $S$. Then comparing the distinguished triangle $\mgbm(V)\to \mgbm(U)\to \mgbm(T)$  (see (\ref{emgys})) with the 'obvious' triangle $\mgs(V)\to \mgs(U)\to \mgs^U(T)$ one can easily prove that $\mgbm(T)\cong  \mgs^U(T)\lan -s\ra$. Yet this isomorphism is not canonical, and so not sufficient for our applications (to Proposition \ref{pcss} below). Still it certainly yields that our definition of $\mgbm(-)$ is closely related to the one from \S I.IV.2.4 of \cite{lemm}. The main disadvantage of defining $\mgbm(T) $ as $ \mgs^U(T)\lan -s\ra$ is that it requires the existence of an embedding of $T$ into a smooth $S$-schemes; yet this would cause no problems in the current paper since we could have restricted ourselves to quasi-projective $S$-schemes (recall here that we have defined $\chow(S)$ above using  projective $S$-schemes only). Moreover, this alternative definition could have nice "integral coefficient properties" and is easier to "lift to a model of $\dms$" (see Remark \ref{rcymod}(6) and \S\ref{sgws}
 below, respectively).

3. 
Actually, above we have demanded $T$ to be reduced just for simplicity of formulations. In particular, 
 assertion II.\ref{iso} is also valid in the case where $T$ is not reduced (certainly here we also 
set  $\mgs^U(T)=\mgs(V\to U)$; cf. Theorem \ref{tcisdeg}(\ref{ifuh})). %

\end{rema}

  \subsection{$\delta$-coniveau spectral sequences for the cohomology of Borel-Moore motives}\label{scss}
  
 
 Now let us study certain ($\delta$)-coniveau spectral sequences for the cohomology of the Borel-Moore motives of finite type $S$-schemes. These properties are quite natural (though somewhat technical).  
For the convenience of readers, we also note now  that in the proof of Theorem \ref{thts}(I) below only part I.1 of  Proposition \ref{pcss} is applied;
cf. also Remark \ref{rcouples}(3) below.

  We introduce some notation.
  
\begin{defi}\label{dcoh} 
\begin{enumerate} 

\item\label{i2} Let $\au$ be an abelian category; let $H:\dmcs\to \au$ be a cohomological functor 
(see \S\ref{snotata}). Then for a 
finite type $X/S$ and $j\in \z$ we define $H^j(X)$ as $H(\mgbm(X)[-j])$. Besides, for an $H\in \obj \dms$ we will denote $\dms(\mgbm(X)[-j], H)$ by $H^j(X)$ (also).

Moreover, we will need certain twists of cohomology theories: for an $l\in\z$ we  denote $X\mapsto H(X\{l\})$ by $H_{-l}$
(following \cite{3}).

\item\label{i3} Assume in addition that $\au$ satisfies AB5 (see \S\ref{snotata}). 
Let $X_0$  be a
scheme essentially of finite type over $S$
 (in particular, the spectrum of 
a 
field over $S$; see \S\ref{snotata}); present $X_0$ as an inverse limit of 
schemes $X_i$ of finite type over $S$ (with the connecting morphisms $j_{ii'}$ being open affine dense embeddings).
Then for any cohomology theory $H$ (including the theories $H^j_l$ as introduced above) we set
  $H(X_0)=\inli H(X_i) (=\inli H(\mgbm(X_i)))$. Here the transition morphisms between $\mgbm(X_i)$ are induced by the transformations 
 $M_{j_{ii'}}$; see Remark \ref{rmgbm}(\ref{ino}). It is easily seen that this direct limit is uniquely determined by  $X_0$ considered as an $S$-scheme; cf. \S3 of \cite{degenmot}.

\item\label{i4} For an $H$ as above, a finite type morphism $a:S'\to S$, and 
any separated $b:S\to S''$ we define $a^!H:\dmcsp\to \au$ (resp. $b_*H:\dmcspp\to \au$)
as $H\circ a_!$ (resp. $H\circ b^*$). Certainly, if $H$ is represented by an  $M\in \obj\dms$ (and $\au=\qv$), then
$a^!H$ (resp. $b_*H$) is represented by $a^!M$ (resp. by $b_*M$); this is the reason for choosing this notation. 

\item\label{i5} Below we will need certain flags of subschemes for a (reasonable) scheme $X$ (of finite type over $S$).

So, we consider a projective system $L=L(X)=\{\lambda\}$
whose elements are sequences of closed  reduced 
subschemes:
a $\lambda\in L$ is a sequence $(X_i^\lambda,-\infty \le 
 i<\infty)$ 
such that   $X^\lambda_i$ is a closed subscheme of  $X^\lambda_{i-1}$ for all $i\in \z$, $\codimz_X X^\lambda_i\ge i$ for all $i\ge 0$ and $X_i=X$ for all $i\le 0$ 
(note that all $X^\lambda_i$ are empty starting from $i=\dimz(X)+1$; yet "formally infinite" sequences of subschemes are more convenient for the arguments below). 
 The  partial ordering of $L$ is given by 
 embeddings 
$X_i^\lambda\to X_i^{\lambda'}$.

We set $U_i^\lambda=X\setminus X_i^\lambda$ (for all $i\in \z$).

\end{enumerate}
\end{defi}

  Now we are able to adjust the standard construction of coniveau spectral sequences to our setting. The following spectral sequence would have looked somewhat more 'familiar' if we would have added the twist $\lan \dimz(X) \ra$ to the definition of $\mgbm(X)$; yet such a modification would have caused certain  notational problems in the case where $\delta$-dimensions of the connected components of  $X$  are distinct. 

\begin{pr}\label{pcss}
I Adopt the conventions (and notation) of 
Definition \ref{dcoh};
 assume that $X/S$ is a  
finite type scheme; let $g:X'\to X$ be a morphism of reasonable schemes. Then the following statements are valid.

1. 
There exists a  spectral sequence with 
\begin{equation}\label{ee1}
E_1^{pq}=\coprod_{x\in \xx^p} H^{q}(x)\end{equation}
 (where $\xx^p$ is the subset of the set $\xx$ of points of $X$ consisting of points of $\delta$-codimension $p$) that converges to $E_\infty^{p+q}=H^{p+q}(X)$ (in the notation of part 2 of the Remark); we will call  spectral sequences of this sort  {\it coniveau} ones.

The induced filtration $F^jH^*$ of $H^*$ is the $\delta$-coniveau one i.e. it is given by $\cup \ke H^*(X)\to H^*(U_{j}^\lambda)$ for 
$U_{j}^\lambda$ (for a fixed $j\ge 0$) running through all open subschemes such that $\codimz_X(X\setminus U_{j}^\lambda)\ge j$. 

Moreover, this spectral sequence (together with the corresponding exact couple) coincides with the $\delta$-coniveau one for the cohomology of $X$ "over $X$" with coefficients in $x^!H$.

2. 
There also exists a 
spectral sequence with 
$E_1^{pq}=\coprod_{x\in \xx^p} H^{q}(x\times_X X')$ 
 that converges to $H^{p+q}(X')$.

3. Assume that $g$ is smooth of  
 relative dimension $s$, whereas $\dimz(X')=\dimz(X)+s$. 
 Then  $H^*(x_!M_g(\mgbm(X)))$ (see Remark \ref{rmgbm}(\ref{ino})  is compatible with a certain natural morphism between $\delta$-coniveau spectral sequences converging to $H^{*-s}_{-s}(X')$ and $H^*(X)$, respectively. Moreover, 
this construction is well-behaved with respect to compositions of smooth morphisms. 
Finally, the corresponding morphisms between $E_1$-terms of coniveau spectral sequences are naturally compatible with the morphisms between points of $X'$ and $X$ (of equal $\delta$-codimensions) induced by $g$.

4. 
Assume that $g$ is finite; denote $\dimz( X)-\dimz( X')$ by $c$. 

Then there exists a natural morphism from the $\delta$-coniveau spectral sequence converging to $H^*(X')$ to the shift of the one for $H^*(X)$
by $c$ (i.e. $E_r^{pq}(X')$ maps to $E_r^{p+c,q}(X)$)  
that is compatible with $H^*(x_!M^g(\mgbm(X)))$. Moreover, 
this construction respects compositions  of finite morphisms.
Finally, the induced morphisms of $E_1$-terms are naturally compatible with the morphisms between points of $X'$ and $X$ (of  $\delta$-codimensions that differ by $c$) induced by $g$.


II 
 For any finite type $y:Y\to S'$ the $\delta$-coniveau spectral sequence for $a^!H(Y)$ (for $Y$ considered as an $ S'$-scheme) is naturally isomorphic to the one for $H(Y)$
(for $Y/S$).


III Assume that $X=X''\times_S S''$ for some finite type $x'':X''\to S''$. Then the following statements are valid.

1. For the $\delta$-coniveau spectral sequence for $b_*H^*(X'')$ we have 
$E_1^{pq}=\coprod_{x\in \xxpp^p} H^{q}(x\times_X X')$, whereas 
 the limit is isomorphic to $H^{q}(X)$.

2. Assume that $b$ is of relative dimension $\le s$ (for some $s\ge 0$),   
and $\dimz( X)=\dimz( X'')+s$. Then  we have a natural morphism from the $\delta$-coniveau spectral sequence  
for $b_*H^*(X'')$ 
to the one for $H^*(X)$.

Moreover,  this morphism is an isomorphism if $s=0$. 

3. More generally, for a proper $b''$ and an $X''$ as above assume that $S$ is of dimension $\le s$ over a subscheme  of $S''$ 
of $\delta$-codimension $c$ (for some $c\ge 0$), whereas $\dimz( X)=\dimz( X'')-c+s$. Then we have a morphism  from the $\delta$-coniveau spectral sequence  
for $b_*H^*(X'')$ 
to the shift of the one for $H^*(X)$
by $c$ (cf. assertion I.4). This morphism is an isomorphism if $s=0$.

\end{pr}
\begin{proof}
I We can assume that $X$ is reduced (see Theorem \ref{tcisdeg}(\ref{ifuh})). 

Now let us apply a more or less standard method for constructing coniveau spectral sequences (see \S3 of \cite{ndegl}; in \S1 of
 \cite{suger} a closely related argument is described; cf. also \S2 of \cite{degfib}). 
 Let $d$ denote  $\dimz(X)$. 

For any $\lambda\in L=L(X)$ (see  Definition \ref{dcoh}(\ref{i5})) we take the sequence of morphisms 
$\dots \to x_!\mgx(U_{-1}^{\lambda}) \to (0=)x_!\mgx(U_0^{\lambda})\to x_!\mgx(U_1^\lambda)\to x_!\mgx(U_2^\lambda)\to\dots\to x_!\mgx(U_d^\lambda)\to x_!\mgx(U_{d+1}^\lambda)=x_!\mgx(X)$ (and we also have $x_!\mgx(U_{i}^\lambda)=x_!\mgx(X)$ for all $i>d+1$). Completing all of these morphisms to distinguished triangles 
 $x_!\mgx(U_{i-1})\to x_!\mgx(U_i)\to x_!\mgx(U_{i-1}\to U_i)$
 yields a so-called {\it 
 Postnikov tower} in $\dmcs$ that we will denote by $P(\lambda)$. 

Applying $H^*$ to this tower 
yields an exact couple with $D_1^{pq}=H^{q-1}(x_!\mgx(U_p^\lambda))=x^!H^{q-1}(\mgx(U_p^\lambda))$, $E_1^{pq}=x^!H^{q}(\mgx(U_p^\lambda\to U_{p+1}^\lambda))$
(see \S3 of \cite{ndegl} or Exercise 2 in \S IV.2
 of \cite{gelman}). We  denote this couple by $EC(X,x^!H,\lambda)$ 
 Note that the spectral sequence corresponding to this couple converges to $E_\infty^{p+q}= D_1^{0,p+q+1}=H^{p+q}(x_!\mgx(X))=H^{p+q}(X)$, and the induced filtration is given by $\ke (H^{p+q}(X)\to H^{p+q}(U_p^\lambda))$.

Next, we take  exact couples as above for all $\lambda$ and pass to the direct limit 
over $L$. We can do so since
our system of couples is "coherent".  
Indeed, all the connecting morphisms in these couples and the transition morphisms between them come from applying $x_!H\circ \mgx$ to $K^b(\smx)$-morphisms (see Proposition \ref{pmg}(I.\ref{ikbsmo}-\ref{ikbsme})).

Our spectral sequence is the one corresponding to this limit couple $EC(X,x^!H,L)$. 
Certainly, this limit spectral sequence also converges to 
$H^{p+q}(X)$, whereas  the induced filtration is the $\delta$-coniveau one. 
Besides, the "moreover" part of the statement is certainly given by construction.

It remains to calculate the $E_1$-terms of this spectral sequence. Standard arguments (cf. \S1.2 and Remark 5.1.3(3) of \cite{suger}; see also Remark \ref{rcouples}(2) below) yield that $E_1^{pq}$ splits as 
$\coprod_{x\in X^p}\inli x^!H^{q}\mgx^{U_x^\al}(T_x^\al)$ where the direct limit is taken for all open neighbourhoods $U_x^\al$ of $x$ in $X$ and $(T_x^\al=\cl_{U_x^\al}x$;
here we apply Proposition \ref{pmg}(II.\ref{iop},
\ref{inis}).
Hence 
Proposition \ref{pmg}(II.\ref{iso}, I.\ref{ikbsme}) yields the 
 formula desired for $\inli_L E_1^{pq}(\lambda)$. 


2. The method of the proof of the previous assertion carries over to this (relative) setting without any difficulty (cf. also Theorem 1.2.1 of \cite{degfib}). 
For $L$ as above denote by $L^{X'}_X$ the "base change of $L$ to $X'$ (i.e. its elements are $\lambda^{X'}_X=(X_i^\lambda\times_X X')$).
Then one considers the exact couples of the type $EC(X',g^!x^!H, \lambda^{X'}_X)$ (i.e. $D_1^{pq}=H^{q-1}(x_!g_!\mgxp(U_p^\lambda\times_X X'))$, $E_1^{pq}=H^{q}(g_!x_!\mgxp(U_p^\lambda\times_X X'\to U_{p+1}^\lambda\times_X X'))$ and then passes to the direct limit with respect to $L^{X'}_X$.
The method of the calculation of $E_1$-terms is also similar to the one above. 

3. 
Proposition \ref{pmg}(I.\ref{ikbsmo},\ref{ikbsme}) easily yields that 
one should map the $\delta$-coniveau spectral sequence converging to $x!H^*\circ g_\sharp (X'))$ ("computed at $X'$")
 into the one for $x!H^*(X)$.
It certainly suffices to construct a natural morphism between the corresponding exact couples.
Moreover, the morphism of spectral sequences obtained will be compatible with $H^*(x_!M_g(\mgbm(X)))$ if we factor the morphism of couples mentioned through the 
 couple for $x^!H^*\circ g_\sharp g^*(X)$. So, for a $\lambda$ as above (see Definition \ref{dcoh}(\ref{i5}))) and for  $L'=\{\lambda'\}$ corresponding to $X'$ one should 
compare  $EC(X,x^!H\circ g_\sharp g^*,\lambda)$ (i.e. $(x_!H^{*-1}(g_\sharp g^*\mgx(U_p^\lambda)),x_!H^*(g_\sharp g^*\mgx(U_p^\lambda\to U_{p+1}^\lambda)))$) with
$EC(X',x^!H\circ g_\sharp,L')= \inli_{\lambda'\in L'} EC(X',x^!H\circ g_\sharp,\lambda')$. 
Now, Proposition \ref{pmg}(I.\ref{ikbsmo}) yields that $g^*\mgx(U_p^\lambda))$ (resp. $g^*\mgx(U_p^\lambda\to U_{p+1}^\lambda))$) equals $\mgxp'(U'{}_p^{\lambda'}))$ (resp. $\mgxp'(U'{}_p^{\lambda'}\to U'{}_{p+1}^{\lambda'}))$. Hence $EC(X,x^!H\circ g_\sharp g^*,\lambda)=EC(X',x^!H\circ g_\sharp,\lambda^{X'}_X)$ (in the notation introduced above). 
Since we have  $L^{X'}_X\subset L'$,  the natural morphism $\inli_{\lambda'\in L^{X'}_X} EC(X',x^!H\circ g_\sharp,\lambda')\to \inli_{\lambda'\in L'} EC(X',x^!H\circ g_\sharp,\lambda')$ yields the morphism of spectral sequences desired.  

It remains to note that the corresponding morphisms of $E_1$-terms are induced by morphisms to points of $X$ of the generic points of their base change to $X'$ (see Proposition \ref{pmg}(II.\ref{iei})).


4. We argue somewhat similarly to the previous proof. 
We consider the following map from $L$ to $L'$ (as considered above): we send $\lambda$ to $\lambda_X^{X'}[c]=(X'_{\lambda_X^{X'}[c],i})$
where $X'_{\lambda_X^{X'}[c],i}=X'$ if $i\le c$ and $=X^\lambda_{i-c}\times_X X'$ otherwise.
Obviously, the image of this map is unbounded 
 in $L'$ (and so, can be used for the calculation of direct limits with respect to $L'$). We also note that 
 for all "large enough" $\lambda\in L$ we have $\imm X'\subset X_i^\lambda$ for all $i\le c$; hence for these $\lambda$ the spectral sequence for  $\lambda_X^{X'}[c]$ differs from the one for  $\lambda_X^{X'}$ exactly by the "shift by $c$" as described in the assertion.
Hence in order to prove the assertion it remains to make the corresponding cohomology calculations.

We have $EC(X,x^!H\circ g_* g^*,\lambda)=EC(X',x^!H\circ g_!,\lambda^{X'}_X)=EC(X',(x\circ g)^!H, \lambda^{X'}_X$; 
this yields the result in question.

Lastly, 
 the corresponding morphisms of $E_1$-terms are induced by $M^g$ applied to $\mgx(V_x^\al\to U_x^\al)$ 
for (the "formal inverse limit of") all open neighbourhoods $U_x^\al$ of a given point $x$ in $X$, and $V_x^\al=U_x^\al\setminus \cl_{U_x^\al}x$.

II Immediate from the "moreover" part of assertion I.1.

III.1.  By construction, the $\delta$-coniveau spectral sequence in question comes from the exact couple  $
EC(X'',x''^!b_*H, L'')$ (where  $L''=L(X)=\{\lambda''\}$ corresponds to $\delta$-coniveau spectral sequences for $X''$; cf. Definition \ref{dcoh}(\ref{i5})).

Let $b_X$ denote the base change of $b$ via $x''$. Then we have an invertible natural transformation   $b^*x''_!\to x_!b_X^*$ (see Theorem \ref{tcisdeg}(\ref{iexch}). This immediately yields the desired calculation for the 
$E_1$-terms of our spectral sequence. Moreover,
 we have  $EC(X'',x''^!b_*H, \lambda'')\cong  EC(X'',b_{X*}x^!H, \lambda'')$, and isomorphisms of this type are compatible with the transition morphisms between the couples of the former type for $\lambda''$ running through $L''$. Hence our spectral sequence is (canonically) isomorphic to that coming from  $\inli_{\lambda''\in L''} EC(X'',b_{X*}x^!H, \lambda'')=\inli _{\lambda''^X_{X''}\in L''^X_{X''}} EC(X,x^!H, \lambda''^X_{X''})$. Hence the formula desired for the $E_1$-terms of the corresponding spectral sequence is an easy consequence of Proposition \ref{pmg}(I.\ref{iii}); cf. the proof of assertion I.1 (and of I.2). 

2.  Our $\delta$-dimension 
 assumptions yield that $L''^{X}_{X''}\subset L$. Hence the 
calculations made above yield the following: the morphism in question is induced by the  obvious morphism $\inli_{L''^{X}_{X''}}EC(X,x^!H,\lambda''^X_{X''})\to \inli_{L''}EC(X,x^!H, \lambda''^X_{X''})$. Lastly, if $b''$ is quasi-finite (and $\dimz( X'')=\dimz( X)$) then  this morphism  is an isomorphism since any element of $L$ is dominated by a one of $L''^{X}_{X''}$ (cf. the proof of assertion I3).

3. It suffices to replace $\lambda''^X_{X''}$ by $\lambda''^X_{X''}[c]$ (see the proof of assertion I.4) in the previous argument. 

\end{proof}

\begin{rema}\label{rcouples}

1. In particular, if $g$ is an open embedding (and $\dimz( X')=\dimz( X)$) then the induced morphisms of $E_1$-terms are the natural split epimorphisms (corresponding to the decompositions of the form (\ref{ee1})). If
$g$ is a closed embedding 
then the induced morphisms of $E_1$-terms are the natural 
 split monomorphisms (see Proposition \ref{pmg}(\ref{icl})).

Certainly, here we compare the decompositions of the type (\ref{ee1}) using the constructions of Proposition \ref{pmg} (see also Remark \ref{rnat}(1)).

2. In \S3 of \cite{ndegl} two closely related methods for constructing coniveau spectral sequences were described; one may say that they correspond to "complementary" Postnikov towers. Above we have considered the exact couple corresponding to the method (very much) similar to the one used in \cite{bger} and \cite{bgern}, whereas in \S1 of \cite{suger} the "complementary" exact couple was considered. Yet the couples mentioned yield canonically isomorphic spectral sequences (see \S3 of \cite{ndegl}); hence the method of calculation of the $E_1$-terms of the spectral sequence in \S1.2 of \cite{bger}  
can be applied to 'our' exact couple also. Besides, it is actually not difficult at all to use the "complementary" exact couples in all the arguments above.

3. It is not very easy to grasp Proposition \ref{pcss} (and especially its proof). Unfortunately, the author does not now how to simplify it. Yet below we will rarely need the 'whole' $\delta$-coniveau spectral sequences; most of the time we will only be interested in their $E_1$-terms and boundary morphisms between them. 

So, for a $\lambda$ as in 
 Definition \ref{dcoh}(\ref{i5}) we consider the complexes $C_t(X,x^!H,\lambda)$ (for $t$ running over all integers) whose terms are $C_t^m(\lambda)=x^!H^{t}(\mgx(U_{-m}^\lambda\to U_{1-m}^\lambda)[m])$ and the boundary morphisms are induced by the obvious $K^b(\smx)$-ones. Then we define the Cousin complex $C_t(X,x^!H,L)$ as the direct limit of  $C_t(X,x^!H,\lambda)$ for $\lambda$ running over $L$. 

By construction, the collection of these complexes yields the $E_1$-level of the spectral sequence described in Proposition \ref{pcss} (so, $C^m_t(X,x^!H,L)=E_1^{-m,t-m} $). Thus one can reformulate those assertions of Proposition \ref{pcss} that mention $E_1$, in terms of $C_t(-)$, whereas the shifts of  the spectral sequences mentioned just yield shifts of Cousin complexes. We also note that replacing the proposition with the corresponding corollary (i.e by its '$C_t(-)$-modification') would not have simplified the proof substantially; yet understanding the proof would have probably become easier.

4. One can also prove a certain analogue of assertion I.4 when $g$ is a proper morphism of arbitrary relative dimension $s$.
To this end one should take $c=\dimz( X)-\dimz( \imm X')$ and compare the exact couples corresponding to $\lambda_X^{X'}[c]$ (for $\lambda\in L$) with 
that for $\lambda'[s]=(\dots, X', X',X'{}_{s+1}^{\lambda'},X'{}_{s+2}^{\lambda'},\dots)$. 
 Note that the spectral sequence corresponding to 
$\inli_{\lambda'\in L'}EC(X',x^!H\circ g_!,\lambda'[s])$ is closely related to the 'usual' coniveau spectral sequence for $(X, x^!H\circ g_!)$. 

5. The author suspects that some analogue of assertion I.3 can be established for an arbitrary flat morphism of regular schemes (using Theorem \ref{tcisdeg}(\ref{ipura})).

\end{rema}

Sometimes we will need $\delta$-coniveau spectral sequences for schemes that are (only) essentially of finite type over $S$.

\begin{coro}\label{cpsc} 
Let $X_0=\prli X_i$ be 
essentially of finite type over $S$ (see \S\ref{snotata}). Then the following statements are valid.

I There exists a spectral sequence with $E_1^{pq}=\coprod_{x\in \xx_0^p} H^{q}(x)$ that converges to $H^{q}(X_0)$ (see Definition \ref{dcoh}(\ref{i3})). 

II Assume that $\dimz(X_i)\le  d$ (for some $d\ge 0$).

1. Suppose 
for any $l\ge 0$ and $x\in \xx_0^l$ we have $H^{q}(x)=0$ for all $q> l-d$. Then $H^{j}(X_0)=0$
for any $j>0$. 

2. Suppose 
for any $l$ and $x$ (as above)  we have $H^{q}(x)=0$ for all $q< l-d$. Then $H^{j}(X_0)=0$
for any $j< -d$.  
\end{coro}
\begin{proof}
I Obviously, it suffices to construct the spectral sequences in question for the connected components of $X_0$; hence we can assume that all $X_i$ are connected.
By Proposition \ref{pcss}(I.3), in this case the direct limit of $\delta$-coniveau spectral sequences converging to $H^*(X_i)$ yields a spectral sequence converging to $H^*(X_0)$.
  It remains to note that 
   the 
  argument that was used in the calculation of $E_1^{**}$ in the proof of Proposition \ref{pcss}(I.1) also easily yields that the corresponding formula  is valid in this case also. 

II Immediate from assertion I.

\end{proof}

\begin{rema}

1. Actually, we could have modified our construction of $\delta$-coniveau spectral sequences in Proposition \ref{pcss}(I.1) so that it would work in this  'limit' setting also; cf. Corollary 3.6.2 of \cite{bger} and Proposition 3.2.4 of \cite{bgern}. 

2. If there exists an $i$ such that  $H^{r+l-d}(x)=0$ for all $r\neq 0$, $l\in \z$, $x\in \xx_i^l$,  then assertion II for this $H$ follows immediately from Proposition \ref{pcss}(I.1).
\end{rema}

We also formulate certain properties of  'stalks' of cohomology at 
fields over $S$.

\begin{coro}\label{cres}
Assume that an abelian category $\au$ satisfies AB5.

I Fix an $X_0$ essentially of finite type over $S$; $j,l\in \z$. 


1. The correspondence $H\mapsto H^j_l(X_0)$ converts coproducts of functors $\dmcs\to\au$ into $\au$-coproducts.

2. Let $F\to G\to H$ be a complex of functors ($\dmcs\to\au$) that is exact in the middle if applied to any object of $\dmcs$. Then 
the corresponding complex $F^j_l(X_0)\to G^j_l(X_0)\to H^j_l(X_0)$ is exact in the middle also.

II For a cohomological $H:\dmcs\to \au$ assume that $H_l^j(X_0)=\ns$ for any $l,j\in \z$ and any 
field $X_0/S$ (see \S\ref{snotata}).
Then the following statements are valid.

1. $H=0$.

2. Assume that $H=\dms(-,E)$ for some $E\in \obj \dms$. Then $E=0$.
\end{coro}
\begin{proof}
I It suffices to note that the correspondence $H\mapsto H^j_l(X_0)$ is the direct limit of the correspondences $H\mapsto H^j_l(X_i)$ (where $X_0=\prli X_i$), whereas for the latter functors both of the assertions are obviously fulfilled.

II.1. Proposition \ref{pcss}(I.1) yields that $H_l^j(X)=0$ for any $l,j\in \z$ and any 
finite type $X/S$. In particular,  $H^*(x_!\q_X(r))=0$ for any $r\in \z$ and any projective morphism $x:X\to S$  (of finite type) such that $X$ is regular. Hence Theorem \ref{tcisdeg}(\ref{igenc}) yields the result.

2. Immediate from the previous assertion by part \ref{imotgen} of the theorem.
\end{proof}

\section{Main results:  perverse homotopy $t$-structures,  cycle modules, and  Gersten weight structures}

In \S\ref{sthom} we define $\thom(-)$ and prove several properties of these $t$-structures (this includes a comparison of $\thom(S)$ with the perverse homotopy $t$-structure of Ayoub).

In \S\ref{scycmod} we relate $\hrthom(S)$ to cycle modules over $S$ (as defined in \cite{rostc}).

In \S\ref{seff} we introduce a certain (new)
effectivity filtration for $\dms$ and study its properties. 
It turns out that that there exists a natural version of $\thom(S)$ for the corresponding $\dmes\subset \dms$ such that the functors $r:\dms\to \dmes$ (right adjoint to the embedding) and $-_{-1}$ (generalizing the functor defined by Voevodsky) are $t$-exact.

In \S\ref{sfunctcss} we discuss the motivic functoriality of coniveau spectral sequences and their description in terms of $\thom(S)$.

In \S\ref{sgws} we sketch the proof of the main properties of the Gersten weight structure on a certain category $\gdos$ of $S$-comotives. In particular, for a {\it countable} $S$ we obtain a new description of $\hrthom(S)$. Besides, the results of this subsection 
 yield natural proofs of the claims made in the preceding one.

   \subsection{The definition  and main properties of $\thom$}\label{sthom}

   Let us define a $t$-structure on $\dms$  using the preaisle that we have already considered in Theorem \ref{tmgeff}.

\begin{defi}\label{dhts}
Consider the preaisle generated by $\{\obj\chowe(S)\{i\},\ i\in \z\}$ (see Remark \ref{rmgbm}(\ref{ichows}) and Definition \ref{daisle}) and denote by $\thom=\thom(S)$ 
the $t$-structure corresponding to it (see Proposition \ref{paisle}).
\end{defi}

Certainly, this $t$-structure depends on the choice of $\delta$ in the obvious way (so, we will not include $\delta$ in the notation).

   Let us establish a characterization of $\thom$ in terms of 'stalks' at 
   fields over $S$; this implies several other interesting properties of this $t$-structure.

\begin{theo}\label{thts}
I. The following statements are valid.

1. Let $H\in \obj \dms$. 
Then we have $H\in \dms^{\thom\le 0}$ (resp. $H\in \dms^{\thom\ge 0}$, resp. $H\in \dms^{\thom=0}$) if and only if 
for $X_0$ running through all 
fields over $S$ we have $H^l_j(X_0)=\ns$ for all $l>-\dimz( X_0)$ (resp. for all $l<-\dimz( X_0)$, resp. for all $l\neq -\dimz (X_0)$; see 
 Definition \ref{dcoh}(\ref{i3}).

2. $\thom$ is non-degenerate.

3. The functor $-\{r\}$ is $t$-exact (with respect to $\thom$) for any $r\in \z$. 

II Let $a:S'\to S$ be a finite type morphism. Then the following statements are valid.

1. $a^!$ is $t$-exact (with respect to $\thom(S)$ and $\thom(S')$).

2. Assume that $a$ is of relative dimension $\le r$. 
Then $a_*\dmsp^{\thom(S')\ge 0}\subset \dms^{\thom(S)\ge 0}$ and $a_*\dmsp^{\thom(S')\le 0}\subset \dms^{\thom(S)\le r}$.

In particular, $a_*$ is $t$-exact if $a$ is quasi-finite. 

III Let $j:U\to S$  be an open immersion, and $i:Z\to S$  be the complementary closed embedding. Then $\thom(S)$ is the $t$-structure glued from $\thom(Z)$ and $\thom(U)$ (see Definition \ref{dglu}(2)). 

\end{theo}
\begin{proof} I.1. It certainly suffices to find out when $E\in \dms^{\thom\le 0}$ and when $E\in \dms^{\thom\ge 0}$.

Theorem \ref{tmgeff}(I.\ref{idecomp}) yields that $\dms^{\thom\le 0}$ can also be described as the preaisle generated by $\mgbms(X)\lan d\ra\{r\}$ for all $r\in \z,\ d\ge 0$, and 
finite type
$X$ over $S$  of $\delta$-dimension $d$. 

Hence for any $H\in \dms^{\thom\ge 0}$ 
 we have $H^l_j(X_0)=\ns$ for  
any field $X_0$ over $S$ 
	of $\delta$-dimension $e$ and all $l<-e$  (since these 'stalks' are direct limits of the corresponding "values" of $H$, whereas the latter vanish by the orthogonality axiom of $t$-structures).

Conversely, if  all such $H^l_j(X_0)=\ns$, then $\mgbm(X)\{r\}[d+1]\perp H$ for any $r\in \z$ and 
finite type $X/S$ of $\delta$-dimension $d$ 
by 
Corollary \ref{cpsc}(II.2). Hence applying Proposition \ref{paisle}(3) we get $H\in \dms^{\thom\ge 0}$.

Now, let $H\in \dms^{\thom\le 0}$. Then $H^l_j(X_0)=\ns$ for any $l>-e$ and any 
field $X_0/S$ of $\delta$-dimension $e$  
 by Theorem \ref{tmgeff}(II.2).

It remains to prove the converse implication (for some $H\in \dms^{\thom\le 0}$).  Let us take the $\thom$-decomposition $$H^{t\le 0}\to H\to H^{t\ge 1}\to H^{t\le 0}[1] $$ of $H$. By Corollary \ref{cres}(II.2) for $G=\dms(-,H^{t\ge 1})$ it suffices to prove that  $G_r^i(X_0)=\ns$ for all $i,r\in \z$ and all $X_0/S$.
 According to the "$\thom$-positive" part of our assertion (that we have just proved) the latter statement is valid if $i\le -\dimz( X_0)$. On the other hand, all the "positive stalks" of $H$ are zero by our assumption and (as we have already proved) the same is true for $H^{t\le 0}[1] $; thus 
 applying Corollary \ref{cres}(I.2) we conclude the proof.

2. Immediate from the previous assertion combined with Corollary \ref{cres}(II.2).

3. Immediate from the uniqueness of recovering $t$-structures from preaisles (see Remark \ref{rts}(3) or Proposition \ref{paisle}) and the fact that   
$-\{r\}$ is an automorphism of $\dms$ that restricts to a bijection on the "generators" of $\dms^{\thom\le 0}$.

II According to assertion I.1, it suffices to compute the  "stalks" of $a^!H$ for $H\in\obj \dms$ and of $a_*H'$ for $H'\in \obj\dmsp$, respectively.

1.
 Let $X_0=\prli X_i$ be a 
field over $S'$; assume that $X_i$ are irreducible. Then the isomorphism of stalks in question is just a part of Proposition \ref{pcss}(II) (applied to $X$ being one of the $X_i$ and to the twists of $H$).

2.  Proposition \ref{pcss}(III.1) yields that the stalk of $a_*H^j_l$ at $X_0/S$ equals $H^j_l(X_0\times_{S}S')$. Hence the statement desired is immediate from Corollary \ref{cpsc}(II).

III By Remark \ref{rglu}(2), it suffices to note that $i_*=i_!$ and $j^*=j^!$ are $t$-exact, whereas the latter facts are given by  assertion II.

\end{proof}

  \begin{rema}\label{rotex}
  1. Since $H\mapsto H^{-e}_l(X_0)$   yields a cohomological functor on $\dms$ that vanishes on $\dms^{\thom\le -1}$ and on $\dms^{\thom\ge 1}$ (for any fixed $l\in \z$ and a 
  field $X_0/S$ of $\delta$-dimension $e$, 
	the restriction of this functor to $\hrthom(S)$ is  exact (as a functor between abelian categories $\hrthom(S) $ and $ \qv$). Moreover, part I.2 of the theorem immediately yields that the collection of all of these functors yields an exact 
	conservative embedding of $\hrthom(S)$ into $\coprod_{l,X_0}\qv$.
  Certainly, this functor is very far from being full; so it does not yield a description of $\hrt$. We will discuss a possible way to "enhance" this functor to a full embedding in the next subsection.  
  
  2. Certainly, part II of the theorem (together with Remark \ref{rts}(3)) also yields that $a_!$ is right $t$-exact, and $a^*(\dms^{\thom(S)\le 0})\subset \dmsp^{\thom(S')\le r}$ for $a$ of relative dimension $\le r$. 
  
	3. In Definition 2.2.59 of \cite{ayobook} a certain analogue of our $\thom(S)$ was defined as follows. Some base scheme (that we will here denote by $S_0$)  for $S$ was fixed, and the $t$-structure corresponding to the preaisle generated by $x_!x^!f^!\q_{S_0}\{i\},\ i\in \z\,$ was considered, where $f:S\to S_0$ is the structure morphism, $x$ runs through all quasi-projective $S$-morphisms.  

Some properties of this $t$-structure were established in \S2.2 of ibid. 
for a general $S_0$. Yet the (full) analogues of  assertions II and III of our theorem were only established under the assumption of the existence of  certain alterations, 
	that are available only in the case where $S_0$ is the spectrum of a field or a discrete valuation ring (see Theorem 2.1.171 of ibid.).

	Moreover, in this situation Proposition 2.2.69 of ibid. yields that it suffices (in order to 'generate' Ayoub's $t$-structure) to consider only those morphisms $x$ that are projective with regular domain. For such an $X$ Theorem \ref{tcisdeg}(\ref{ipur},\ref{ipura})  easily yields that  $x_!x^!f^!\q_{S_0}\cong \mgbm(X)\lan \dimz X-\dimz S_0\ra$. Hence in this case our $\thom(S)$ differs from Ayoub's perverse homotopy $t$-structure only by 
a shift (that depends on the choice of $\delta$).	
	
  4. The previous remark together with   Corollary 2.2.94 of \cite{ayobook} yields the following: 
 if $S$ is the spectrum of a perfect field $F$, $\thom(S)$ (essentially) coincides with the 'usual' homotopy $t$-structure on $\dms$ considered in ibid. The latter $t$-structure was also considered in (Proposition 5.6 of) \cite{degmod}, whereas
  Theorem 5.11 
  of  ibid. (along with Lemma 5.5 of ibid. and part I.1 of our theorem) yields another proof of 
		this coincidence of $t$-structures result. Next, the $t$-structure studied in \cite{degmod}  is compatible with the  Voevodsky's homotopy $t$-structure on $\dme(F)$ by 5.7(1) of ibid. These observations together with Theorem \ref{thts}(II) justify the name "perverse homotopy $t$-structure" for $\thom(S)$ (yet note that the author did not check whether $\thom(S)$ essentially 
		coincides with the Ayoub's $t$-structure for a 'more general' $S_0$).

5. One can easily see (cf. the beginning of  \S\ref{sbmp}) that our proof Theorem \ref{thts} can be generalized to a more general setting of {\it triangulated motivic categories} (cf. 
\cite{degcis}). So, one should consider an arbitrary assignment $X\mapsto \dm(X)$ for $X$ running through finite type $S_0$-schemes 
that satisfies all properties listed in Theorem \ref{tcisdeg} expect (possibly) part \ref{ivoemot}, and it suffices to demand part \ref{ivanmc} of the theorem to be fulfilled only in the case where $X$ is of finite type over a field.

In particular, instead of motives with rational coefficients we could have considered motives whose coefficient ring $R$ is an arbitrary (commutative unital) $\q$-algebra. One can also prove similar results 
 for other $R$ provided that certain analogues of part \ref{imotalt} of the Theorem are known over the corresponding base scheme (it suffices to replace $R$ by $R_{(l)}$ in it, where $l$ runs through all primes that are not invertible in $l$; cf. \cite{degchz}, \cite{bzp},  and \cite{kellymotcharp}).
  
\end{rema}

   
   \subsection{On the relation to cycle modules}\label{scycmod}
   
   Let us relate the objects of $\hrthom$ to {\it cycle modules} over $S$ as defined by Rost. To this end we study the "stalks" of a 
   cohomological functor $H:\dmcs\to \au$ and their "residues" (with respect to $S$-valuations). 

   \begin{defi}\label{dval}
For $X_0$ being the spectrum of a field $F$ over $S$ (see \S\ref{snotata}) we will say that $v$ is an {\it $S$-valuation of $F$} if it is a discrete (rank one) valuation on $F$ that yields a 
codimension $1$ subscheme of some irreducible finite type $X/S$ such that $F$ is the fraction field of $X$. 

We denote the spectrum of the residue field $F_v$ of $v$ by $X_v$.

If $F$ is a residue field of a 
 finite type $X'/S$, 
we will call $v$ an $X'$-valuation of $F$ if it yields a  subscheme of $X'$ of 
codimension $\codim_{X'}X_0+1$ (so, one may say that we are only interested in valuations {\it of the first kind}; cf.  Definition 8.3.18 and the whole \S8.3.2 of \cite{liu}). 

\end{defi}

\begin{pr}\label{pcycmod}
Let $H:\dmcs\to \au$ be a cohomological functor (for an AB5 category $\au$) and let $X_0$ be the spectrum of a 
field $F$ over $S$; denote $\dimz(X_0)$ by $e$; $l$ will denote an integer.
Then the following statements are valid.

I.1. We have a natural graded action of the graded Milnor $K$-theory algebra $\bigoplus_{l\ge 0} K^l_M(F)\otimes \q$ on the object $\coprod_{l\in \z} H^{-e}_l(X_0)$. 

2. For a morphism $m:X'_0\to X_0$ of spectra of $S$-fields  denote the corresponding morphism of fields by $m'$. Then for any $k\in K_M(F)$ we have
$H(m)\circ (k\times_{X_0} -)=m'(k)\times_{X_0'} H(m)(-)$, where $\times$ denotes the 
action introduced in the previous assertion.

II 
For a valuation $v$ of $F$ choose some $X$ 
as in Definition \ref{dval}, and define $\partial_{v,l}:H^{-e}_l(X_0)\to H^{1-e}_{l}(X_v)$ as the corresponding component of the boundary morphism  $E_1^{\dimz(X)-e,-e}\to E_1^{\dimz (X)-e+1,1-e}$ in the $\delta$-coniveau spectral sequence converging to $H^{*}_l(X)$ 
(
recall that the latter morphism is also a differential of the  Cousin complex $C_{-\dimz( X)}(X,x^!H_l,L)$ mentioned in  Remark \ref{rcouples}(3)).

Then the following statements are valid.

1. $\partial_{v,l}$ does not depend on the choice of $X$. 

2. For any finite type $X''/S$ containing 
$(X_0,X_v)$ (i.e. containing the spectrum of the discrete valuation ring for $v$) the corresponding component of the boundary in the $\delta$-coniveau spectral sequence converging to $H^*_{l}(X'')$ 
equals $\partial_{v,l}$. 

3. Let $X_1$ be a points of $X$ such that  $\dimz( X_1)=e-1=\dimz(X_0)-1$, 
 but $X_1$ is not the spectrum of any residue field of $F$ (for an $X$-valuation). 
Then for any $l\in \z$ the morphism  $H^{-e}_l(X_0)\to H^{1-e}_{l}(X_1)$ corresponding to the $\delta$-coniveau spectral sequence for $H^*_l(X)$, is zero. 

III Assume that  $\au=\qv$; let $X_0$ be a point of a finite type $X/S$, $m\in H^{-e}_l(X_0)$.

1.  The set of  $X$-valuations of $F$ 
such that  $\partial_{v,l}(m)\neq 0$, is finite.

2. 
$\sum_{v_0,v_1}\partial_{v_1,l}\circ \partial_{v_0,l}(m)=0$. Here we take the sum over the finite set of pairs $(v_0,v_1)$ such that $v_0$ is an $X$-valuation of  $F$,  $v_1$ is an $X$-valuation of $F_{v_0}$, $\partial_{v_1,l}\circ \partial_{v_0,l}(m)\neq 0$; we consider this sum as an element of $\coprod_{x\in X^{\codimz (X_0)+2}}H_l^{2-e}(x)$. 

\end{pr}
\begin{proof}
I.1. Similarly to \S5 
 of \cite{degenmot}, we note that it suffices to compute the morphisms between the formal pro-objects $\mgbm(X_0)[e]\{l\}$ (cf. \S4.1 of ibid.). The latter are defined as the formal inverse limits $\prli \mg(X_i)[e]\{l\}$ where $\prli X_i=X_0$ (certainly, this pro-object does not depend on the choice of the system $X_i$).
Theorem \ref{tcisdeg}(\ref{imotfun},\ref{itate},\ref{icont}) easily yields 
 $$\begin{gathered} \pdmcs(\mgbm(X_0)[e]\{l\}, \mgbm(X_0)[e]\{l'\})\cong \inli \dmcs \mgbm(X_i)\{l\}, \mgbm(X_i)\{l'\})\\ \cong \inli \dmc(X_i)(\q_{X_i}\{l\},\q_{X_i}\{l'\})\cong  \dmc(X_0)(\q_{X_0}\{l\},\q_{X_0}\{l'\}).\end{gathered}$$ Now, if $F$ is perfect then  the latter group is isomorphic to $K^{l'-l}_M(F)\otimes \q$ by assertion \ref{ivoemot} of the theorem and the well-known relation between motivic cohomology and Milnor $K$-groups (see Theorem 4.3.4 of \cite{degenmot} 
  proved by Suslin and Voevodsky). In order to establish this isomorphism in general (via replacing $F$ by its perfect closure) it suffices to apply part \ref{ifuh} of the theorem (note that passing to the perfect closure does not affect $K^{*}_M(F)\otimes \q$).

2. It suffices to note that the arguments above are natural with respect to morphisms of $S$-fields. 

II.1. It certainly suffices to compare the versions of $\partial_{v,l}$ corresponding to $X'$ with that for its open (dense) subvariety $X''$. 
Hence (by Remark \ref{rcouples}(3))
 the statement is immediate from 
the corresponding "functoriality" of the splitting (\ref{ee1}) (see Proposition \ref{pcss}(I.3) and Remark \ref{rcouples}(1)).

2. 
As we have just noted, 
Proposition \ref{pcss}(I.3)
yields that the corresponding component of the boundary morphism does not change when we pass from $X$ to its open subscheme. Thus it remains to note that part I.4 of the proposition
  yields a similar compatibility  for closed embeddings. 

3. Denote by $X'$ the closure of $X_0$ in $X$. Our assertion follows (again) from the compatibility of coniveau spectral sequences converging to 
$H_l^*(X')$ and to $H_l^*(X)$, respectively. 

III.1. This is just the condition for the set of all $\partial_{v_0,l}$ to define a map $H^{-e}_l(X_0)\to\coprod_{x\in X^{\codimz( X_0)+1}}H_l^{1-e}(x)$.

2. It suffices to note the following: assertions II.3 and III.1 yield that $\sum_{v_0,v_1}\partial_{v_1,l}\circ \partial_{v_0,l}(m)$ equals the corresponding value of the square of the boundary for  $C_{-\dimz( X)}(X,x^!H_l,L)$.
\end{proof}

\begin{rema}\label{rcymod}

1. Certainly, for a general $\au$ an  action  of $\bigoplus_{l\ge 0} K^l_M(F)\otimes \q$ on the object $\coprod H^{-e}_l(X_0)$ is just a  homomorphism from
 $\bigoplus K^l_M(F)$ into the endomorphism algebra of $\coprod H^{-e}_l(X_0)$.
 
 One can also generalize assertion III.1 of the proposition to the case of a general $\au$; to this end instead of an element $m$ of $H^{-e}_l(X_0)$ 
one can consider a morphism to  $H^{-e}_l(X_0)$ from a compact object of $\au$. 

2. Certainly, we could have considered $\coprod_{l\in \z,X_0/S} H^{-\dimz(X_0)+i}_l(X_0)$ for any $i\in \z$. Yet we are mostly interested in those $H$ for which
all $H^{-\dimz(X_0)+i}_l(X_0)$ (and so, the whole Cousin complexes  $C_{i-\dimz( X)}(X,x^!H_l,L)$) vanish if $i\neq 0$ (we could have called an $H$ of this sort a {\it pure} cohomology theory; cf. Definition 2.3.3 of \cite{bgern}). Certainly, if $H$ is (represented by) an object of $\dms$, the latter condition is equivalent to $H\in \dms^{\thom=0}$; the functor $H_0^{\thom}$ produces such a theory from any $\dms$-representable one. For a general $H$ one can construct a similar "truncation" for it using {\it virtual $t$-truncations} with respect to the {\it Gersten weight structure}; see \S\ref{sgws}
 below.
 
 3. We conjecture that the functor $H_{cm}:H\mapsto (\coprod_{l\in \z,X_0/S} H^{-\dimz(X_0)} 
 _l(X_0),\coprod_{v,l}\partial_{v,l})$ yields an equivalence of $\hrthom(S)$ with a certain abelian category of cycle modules. In the case where $S$ is the spectrum of a perfect field this conjecture (and even its integral coefficient version) is given by Theorem 5.11 of \cite{degmod}; Theorem 4.1. of \cite{degchz} extends the latter result 
to varieties over  characteristic $0$ fields.
 
  In Proposition \ref{pcycmod} we have 
  constructed for the elements of $H_{cm}(\dms^{t=0})$ the structure maps mentioned in parts D3,4 of Definition 1.1 of \cite{rostc}. Constructing the maps mentioned in parts D1,2 of loc. cit. is quite simple. Next, the (analogues of) axioms R1a,b of loc. cit. are quite easy in our setting. Also, above we have proved
 the natural analogues of axiom R2a of loc. cit., as well as axioms FD and C of Definition 2.1 of ibid. Possibly, some of the remaining axioms of ibid. have to be adjusted for our setting (this could include multiplying some of our maps by $-1$), and some additional conditions should possibly be invented (if we want to obtain an equivalence of categories, i.e., a full description of $\hrthom(S)$). One of possible difficulties for axiom checks could be the  studying of $\partial_{v,l}$ 
 in the case where $F$ and $F_v$ do not lie over a single point of $S$ (especially if their characteristics are distinct).
 On the other hand, one can possibly reduce our conjecture to the case where $S$ is the spectrum of a 
  discrete valuation ring $R$ via an application of Theorem \ref{thts}(II,III); moreover, it could be sufficient to consider only the case where $R$ is a Henselian valuation ring.
 
 Lastly 
 we note that some of the consequences of the axioms of ibid. 
 seem to be quite easy to prove in our situation (avoiding those axioms of Rost that we do not prove here); cf. parts R2d,e of his Definition 1.1. 
 
 4.  Parts II.2--3 of the proposition immediately yield that  $H_{cm}(H)$ for an $H\in \dms^{\thom=0}$ 
allows the computation of $H^*_*(X)$ as the cohomology of  the Cousin complex  $C_{-\dimz( X)}(X,x^!H_l,L)$ (for any $X$ of finite type over $S$; cf. also the proof of part III.2 of the proposition).

5. Note also that for any $H\in \dms^{\thom=0}$ and any open dense embedding $j:U\to X$ of 
 finite type $S$-schemes of $\delta$-dimension $d$ the corresponding homomorphism $H^{-d}(j)$ (see 
Definition \ref{dcoh}(\ref{i2})) is obviously injective. Certainly, this is also true for any pure cohomology theory $H:\dms\to \au$ (see part 2 of this remark). More generally, $H^{i-d}(j)$ is  injective (resp. bijective)  for any pure $H$ whenever $X\setminus U$ is of $\delta$-codimension $>i$ (resp. $>i+1$) in $X$ (for an $i\in \z$, and $j$ as above).

 6. It would certainly be interesting to extend (at least, some of) the results of the current paper to motives with integral coefficients (or maybe one should 
 take a coefficient ring 
 containing the inverses to all of the positive residue field characteristics for some fixed base scheme $S_0$ chosen; cf. \cite{bzp}). Yet the author does not know which version of $S$-motives would be most appropriate for this purposes (in particular, which of them is related to cycle modules with integral coefficients). F. D\'eglise has suggested considering Morel's $\tau$-positive part of $SH(S)$ 
  or some its analogue (see Theorem 2.16 of \cite{degchz}) in this context. 
 

\end{rema}

  \subsection{The $\delta$-effectivity filtration for $\dm(S)$}\label{seff}
  
In the case where $S$ is the spectrum of a perfect field $k$, Voevodsky (in \cite{1}) considered the  category $\dme$ of so-called motivic complexes.
This is 
'almost' the subcategory of $\dms$ generated by $\chowes$ as a localizing subcategory; $\thom(S)$ has a natural analogue $\thome(k)$  for it (that was defined earlier than the 'stable' version). 
 It turns out that some of the nice properties of $\dme$ and $\thome(k)$ can easily carried out to our relative setting (though their proofs are quite distinct from the 'old' ones over perfect fields). 
 
Let us recall the following well-known 
 fact.

\begin{pr}\label{pbous}
Let $\cu$ be a triangulated category such that the coproduct of any set of objects of $\cu$ exists in it; let $B\subset \obj \cu$ be a set of compact objects.
 Denote by $\du$ the subcategory of $\cu$ generated by $B$ as a localizing subcategory.

Then the inclusion $\du\subset \cu$ admits an exact right adjoint. 
\end{pr}
\begin{proof}
By Theorem 8.3.3 of \cite{neebook} the category $\du$ satisfies the following Brown
 representability condition: 
  every cohomological functor from $\du$ to $\ab$ that converts $\du$-coproducts into products of abelian groups is representable in $\du$. 
  The latter statement easily yields 
  the result (cf. the proof of Proposition 9.1.19 of ibid.).  
 
\end{proof}

Now we 
define a certain category
$\dmes$ and relate it to the perverse homotopy $t$-structure.

\begin{pr}\label{pdme}
Denote by $\dmes$ the subcategory of $\dms$ 
generated by $\chowes$ as localizing subcategory. 

Denote by 
$\thome=\thome(S)$ the $t$-structure corresponding to the preaisle generated by $\{\obj\chowe(S)\{i\}[j],\ i\ge 0,\ j\ge 0\}$ (see Definition \ref{daisle} and Proposition \ref{paisle}(1)).

Then the following statements are valid.

1. Let $H\in \obj \dmes$. 

Then we have $H\in \dms^{\thome\le 0}$ (resp. $H\in \dmes^{\thome\ge 0}$, resp. $H\in \dms^{\thome=0}$) if and only if 
for $X_0$ running through all 
fields over $S$ we have $H^l_j(X_0)=\ns$ for all $l>-\dimz( X_0)$ (resp. for all $l<-\dimz( X_0)$, resp. for all $l\neq -\dimz( X_0)$; see 
Definition \ref{dcoh}(\ref{i3})) and all $j\le -\dimz( X_0)$. 

2. $\thome$ is non-degenerate.

3. There exist exact right adjoints $r$ 
and $r^{1}$ to the embeddings $\dmes\to \dms$ and $\dmes(1)\to \dmes$, 
respectively.

4. $r$ (resp. the functor $-_{-1}:\dmes\to \dmes$; $M\mapsto M_{-1}=r^1(M)\{-1\}$) 
 is $t$-exact with respect to $\thome$ and $\thom$ (resp. with respect to $\thome$).

\end{pr}
\begin{proof}
1,2: the method of the proof of Theorem \ref{thts}(I.1--2) carries over to this "$\delta$-effective" setting without any difficulty.

3. Immediate from Proposition \ref{pbous}. 

4. Immediate from assertion 1 (recall also Theorem \ref{thts}(I.1) in order to verify the $t$-exactness of $r$).
\end{proof}
 
 \begin{rema}\label{rdme}
1. 
It is easily seen that our definition of $\dmes$ coincides with (the rational coefficient version of) the one considered in \cite{degmod} if  $S$ is the spectrum of a perfect field (and $\delta$ is the Krull dimension function). 

So, our assertion 4 corresponds to Theorem 5.3 and 
5.7(1) of 
ibid. (where the corresponding facts were proved for motives with integral coefficients over a perfect field), whereas assertion
1 is the natural analogue to Proposition 5.6 of ibid. (cf. also Lemma 4.36 and Theorem 5.7 of \cite{3}). 

2. Obviously, the filtration of $\dmes$ given by $\dmes(i)$ (for $i\in\z$) is not exhausting (i.e. there exist $S$-motives neither of whose twists are $\delta$-effective). Quite probably it is also not separated (i.e., there exist "infinitely effective" motives). 

On the other hand, 
the restriction of this '$\delta$-effectivity filtration' to $\dmcs$ is obviously exhausting. It also seems that the study of the functoriality properties of our effectivity filtration (that could be quite interesting for itself) would yield that this restriction is separated (by Noetherian induction on 
$S$).

3. One could also have considered for $\dmes$ the  effective version of the Chow $t$-structure $\tcho$ for $\dms$ defined in \S A2 of \cite{brelmot}, that is given by the preaisle generated by  $\{\chowe(S)[j],\  j\ge 0\}$ (as a localizing subcategory). It is easily seen from the corresponding "Chow" analogue of part 3 of the proposition 
 that the functors $r$ and $-_{-1,\chow}:M\mapsto r^1(M)(-1)[-2] $ are $t$-exact with respect to these Chow $t$-structures.
This seems to imply part 4 of our proposition easily.


4. In \cite{degcis}  also certain effective motivic categories over $S$ were constructed; they map to $\dms$, and so one could define the corresponding effective parts of $\dms$ as the localizing subcategories generated by the images of these functors. 
 All of these effective subcategories coincide and are generated by $\mgs(X)$ for $X$ smooth over $S$. Now, assume that $S$ is (everywhere) of $\delta$-dimension $s\ge 0$. Then our $\dmes$ contains    $\mgs(X)(s)$ for all smooth $X/S$; yet 
 the author believes that  it is not generated by  objects of this type (and so, it is bigger and cannot be described in terms of $\mgs(-)$). Thus, if we define $DM^{\delta}_{bir}(S)$ as the Verdier quotient $\dmes/\dmes(1)$, there will exist a natural comparison functor $DM^o_{gm}(S)\to DM^{\delta}_{bir}(S)$ where  the categories $DM^o_{gm}(-)$ are the ones considered in \S5 of \cite{bososn} (and generalizing the ones defined in \cite{kabir}); yet this functor is not invertible for a general $S$.

\end{rema}
  

  \subsection{On the functoriality of  $\delta$-coniveau filtrations and spectral sequences, and their relation to $\thom$}\label{sfunctcss}
  
  Now let us relate $\delta$-coniveau spectral sequences and filtrations to those coming from $\thom$. To the opinion of the author, the most natural method for proving the following proposition uses the Gersten weight structure on $S$-comotives (that we will discuss in the next subsection); yet a part of this statement can be obtained via "more elementary" arguments.
  
  \begin{pr}\label{pcfilt}
  Let $X,X'$ be finite type $S$-schemes of $\delta$-dimensions $d$ and $d'$, respectively.
  
  I Let $H:\dmcs\to \au$ be a cohomological functor, and let $\au$ be an AB5 abelian category.
  
Then for the $\delta$-coniveau spectral sequences and filtrations introduced in Proposition \ref{pcss}(I.1) the following statements are valid.

1. The 
$\delta$-coniveau filtration $F^jH^*$ on $H(\mgbm(-))$ is $\dmcs$-functorial in the following sense: for any 
$h\in \dmcs(\mgbm(X'),\mgbm(X))$, $j\in \z$, the morphism $H(h)$ 
maps $F^jH(X)$ into $F^{j+d'-d}H(X)$.

2. Any $h$ as above is compatible with a certain morphism of the corresponding $\delta$-coniveau spectral sequences ("with a shift by $d'-d$"). Moreover, there exists a method for making these morphisms of $\delta$-coniveau spectral sequences 
 (canonical and) functorial starting from $E_2$.

II. Assume (by our usual abuse of notation) that $H=\dms(-,H)$ for some $H\in \obj \dms$. Then the following statements are valid also.

1. $F^jH(X)=\imm (H^{\thom\le j-d}(X)\to H(X))$. 

2. Starting from $E_2$ the $\delta$-coniveau spectral sequence for $H^*(X)$ can be expressed in terms of the spectral sequence $E_{1}^{pq}=(H^{\thom=q})^p(X)\implies H^{p+q}(X)$ (cf. Theorem 2.4.2(II) of \cite{bger}).

\end{pr}
\begin{proof}[A sketch of the proof of some  of the statements]

I. Certainly, it suffices to compare the corresponding exact couples.

Hence in order to establish assertion 1 and the first part of assertion 2 it suffices to verify the following fact:
for any $X_i^{\lambda}$ as in the proof of Proposition \ref{pcss}(I.1) there exists a choice of ${X'}_i^{\lambda'}$ such that $h$ extends (non-uniquely) to a morphism of Postnikov towers 
$P(\lambda')\to P^{d'-d}(\lambda)$; here the latter tower is obtained from $P(\lambda)$ via shifting the sequence of $U_i^\lambda$ by $d'-d$. 
The latter fact can be (more or less) easily deduced from Theorem \ref{tmgeff}(I.2,II.2) via an argument similar to the proof of the "weak functoriality" of weight decompositions and weight Postnikov towers (in Remark 1.5.9(1) of \cite{bws}).


II 
Assertion 1 and a part of assertion 2 can easily be established using 
some arguments from the proofs of Theorems 4.4.2(6) and 7.4.2 of \cite{bws}, respectively. The whole statement can be established via the method of the proof of Theorem 4.1 of \cite{ndegl}.

\end{proof}

   \subsection{Comotives over $S$, the Gersten weight structure for them, and generalized $\delta$-coniveau spectral sequences}\label{sgws}
   
   Now we sketch the proof of the existence a {\it Gersten weight structure} $\wger(S)$ on a certain triangulated category of comotives over $S$. 
   The existence of this weight structure gives a natural proof of Proposition \ref{pcfilt} and also so-called {\it virtual $t$-truncations} of cohomology theories on $\dmcs$ (see \S2.3 of \cite{bger}). We do not give  detailed proofs of the properties of  $\wger(S)$ here since the applications mentioned do not seem to be extremely interesting.
   Yet note that we also obtain a certain new 'description' of $\hrthom(S)$ when $S$ is {\it  countable} i.e.  if all of its 
    affine subschemes are the spectra of countable rings (i.e. if $S$ has a Zariski cover by the spectra of some countable rings).
   
   So, we will only explain here that the methods of \cite{bgern}  work in our (relative) settings; all the necessary definitions and detailed arguments can be found in ibid. (whereas using the list 
   of definitions in \S1.1 of ibid. can help in locating them). 
   Since the author did not write down the full proofs of the claims made below (yet), it makes some sense (for a cautious reader) to assume 
   that $S$ is countable.
      Indeed, for a countable $S$ the homotopy limits mentioned below are certainly countable; so that one can "define them on the triangulated level" and apply the 'crude' methods of \cite{bger} (that 'almost do not depend on models for motives'; also pay attention to Remark 2.1.2(3) of \cite{bgern}); see also Remark \ref{rwger}(2,3) below. 
   
  As a starting point of our constructions we need a triangulated category $\gdos$ that is closed with respect to all small products, and contains $\dmcs$ as a full cogenerating subcategory of cocompact objects. It order to construct $\gdos$,  let us recall that $\dms$ certainly has a model (i.e., it is the homotopy category of a certain Quillen closed proper 
  stable model category; several possible models of this sort were constructed in \cite{degcis}). We can assume that this model is injective, i.e., that any object is fibrant in it. Moreover, there exists a model $Mo_S$ of $\dms$ such that the functor $\mgs$ (see Proposition \ref{pmg}) comes from a certain functor $C^b(\sms)\to Mo_S$, and the functor $-(-1)$ lifts to an invertible endofunctor of $Mo_S$. 
  
   So, it is no problem to construct   $\gdos$ in question using the methods of \S5.2 of \cite{bgern} (that heavily rely on the results of \cite{tmodel}). $\gdos$ has several nice properties; in particular, its objects are certain homotopy limits of objects of $\dms$ and the Tate twist lifts (from $\dmcs$) to an exact autoequivalence of $\gdos$.
   Next, using the results of \S2.2 of \cite{bgern} one can construct on $\gdos$ a weight structure "cogenerated" by $\cup_{i\in \z}\obj\chowes\{i\}$.
   
   The heart of $\wger(S)$ is cogenerated by (twisted and shifted) comotives of fields over $S$. For such an $X_0/S$ 
   one can easily   construct the corresponding $\mgbm(X_0)\in \obj \gdos$ as follows: consider a smooth connected $S$-scheme $X$ of relative dimension $s$ 
   such that $X_0$ is a point of $X$ and consider the homotopy limit of 
    $\mgs^{U_x^\al}(T_x^\al)\lan -s \ra$ where the direct limit is taken for all open neighbourhoods $U_x^\al$ of $X_0$ in $X$ and $T_x^\al=\cl_{U_x^\al}X_0$ (note that the connecting maps between  various $\mgs^{U_x^\al}(T_x^\al)$ "lift to the model level"!). 
    Arguing similarly to Proposition \ref{pmg}(II.\ref{iop},\ref{icl}) 
     one can prove that the object constructed this way does not depend on the choice of $X$ (at least, up to an isomorphism). Moreover, one can similarly define $\mgbm(Y)$ for any localization $Y$ of a finite type $S$-scheme, as well as the homotopy limit (with respect to $\lambda$ running through $L$) of the Postnikov towers $P(\lambda)$ used in the construction of coniveau spectral sequences in the proof of Proposition \ref{pcss}(I.1).   We will denote homotopy limits of this type just by $\prli$.
     
     Now we are able to formulate the main properties of $\wger(S)$.

\begin{pr}\label{pwger}

\begin{enumerate}
\item\label{idwger} $\gdos_{\wger(S)\le 0}={}^\perp(\cup_{i\in \z}\chowes\{i\}[1])$. 

\item\label{itwe} $\{j\}$ is $\wger(S)$-exact for any $j\in \z$.

\item\label{imgneg} $\mgs(Y)\in \gdos_{\wger(S)\ge -e}$ for $Y$ being a 
an essentially finite type $S$-scheme of $\delta$-dimension $e$.

\item\label{iheart} The heart $\hwgers$ is equivalent to the Karoubization of the category of all  
$\prod_i \mgbm(X_0^{i})(j_i)[j_i+d_i]$ for $X_0^i$ running through  fields over $S$, $d_i=\delta(X_0^i)$,
 $j_i\in \z$.

\item\label{iwpost} For $X$ of finite type over $S$, $\dimz( X)=d$, 
the tower $\prli_{\lambda\in L}P(\lambda)[d]'$
(we simply shift all the objects of the tower by $[d]$; 
 see the discussion preceding this proposition for the construction of the tower) is a weight Postnikov tower for $\mgbm(X)[d]$.

\item\label{idual} For any  cohomological functor 
$H':\dmcs\to \au$ ($\au$ is an AB5-category) we consider its extension to a functor $H:\gdos\to \au$ via the method mentioned in Proposition 4.1.1 of \cite{bgern}.
In particular, 
 we 
do so for the functor $\dms(-,M)$ (from $\dmcs$ to $\qv$) for all   
 $M\in \obj \dms$.

Then the collection of these functors yields a nice duality $\Phi_S:(\gdos)^{op}\times \dms\to \qv$ such that $\wgers\perp_{\Phi_S}\thom(S)$.

\item\label{iwgs} For any extended cohomological functor $H:\gdos\to \au$ the weight spectral sequence $T=T(M,H)$ ($M\in \obj \gdos$) corresponding to $\wgers$ is functorial in $H$ and is $\gdos$-functorial in $M$ starting from $E_2$. In the case where $M=\mgbm(X)$ for a finite type $X/S$ of $\delta$-dimension $d$, one can choose $T$ to be the 
shift of the corresponding $\delta$-coniveau spectral sequence by $d$ (see Proposition \ref{pcss}(I.1,4)) 
 
We will call such a $T$ a {\it generalized $\delta$-coniveau spectral sequence} (as we also did in \cite{bgern}). 

\item\label{idualp} The category of pure extended cohomological functors from $\gdos$ to an abelian $\au$ (satisfying AB5) is naturally equivalent to the category of those contravariant additive functors $\hwgers\to \au$ that convert all $\hwgers$-products into coproducts.

\item\label{idualt}  For any $N\in \obj \dms$ the generalized $\delta$-coniveau spectral sequences for the functor $\Phi_S(-,N)$  can be $\gdos$-functorially expressed in terms of the $\thom(S)$-truncations of $N$ starting from $E_2$ (cf. Proposition 2.4.3 of ibid.).



\end{enumerate}

\end{pr}
\begin{proof}

Assertion \ref{idwger} is essentially just the definition of $\wger(S)$; it immediately implies assertion \ref{itwe}.

Assertion \ref{imgneg} can be proved via the method of the proof of Proposition 3.2.6 of \cite{bgern} (cf. also the proof of Proposition \ref{pcss}(I.1) above).
Assertion \ref{iheart} can be deduced from it 
similarly to the proof of Theorem 3.3.1(7) of \cite{bgern}. It easily yields assertion \ref{iwpost} (cf. ibid.).

The natural analogues of the remaining assertions are Propositions 4.4.1(I.1,II), 4.3.1,  4.6.1, and Corollary 4.4.3  of ibid. respectively; their proofs can be carried over
to our setting without any problems.

\end{proof}

\begin{rema}\label{rwger}

\begin{enumerate}
\item It could be quite interesting to 
construct various connecting functors between  $\gd(-)$ and study their weight-exactness. Most of these functors should  be extensions of 
   certain $\dmc$-ones;  so one needs certain lifts of the 
   latter  to the model level.

Possibly, the author will study this matter in future. He believes that all the functors of the type $f_!$ (for a finite type $f$) as well as $g^*$ for a quasi-finite $g$ are weight-exact (cf. Theorem \ref{thts}(II)).

\item We conjecture that the following analogue of Corollary 4.6.3 of \cite{bgern} is valid: $\hrthom(S)$ is equivalent to the category of those contravariant additive functors $\hwgers\to \qv$ that convert all $\hwgers$-products into coproducts. 

The problem here is to prove that any such additive functor is representable (via $\Phi_S$) by an object of $\dms$ (since in contrast to loc. cit. it usually 
does not come from a sheaf for any Grothendieck topology for $S$-schemes). Still in the case where $S$ is countable (in the sense described in the beginning of this subsection; if this is the case, then morphisms sets and the number of isomorphism classes  of objects in $\dmcs$ are certainly countable) then the representability assertion in question certainly  follows from 
of Theorem 5.1 of \cite{neebrh} that yields that {\bf any} cohomological functor $\dmcs\to \qv$ is representable by an object of $\dms$.  

Our conjecture could be useful for the study of $\hrthom$ and its relation to cycle modules.


\item 
Theorem 5.1 of \cite{neebrh} also simplifies the construction of "comotivic" functors (between various $\gd(-)$) over countable base schemes; in particular, it 
could be applied to the study of Borel-Moore motives of localizations of finite type $S$-schemes.

\item The proof of Proposition \ref{pwger}(\ref{iwgs}) certainly 
uses assertion \ref{iwpost} of the proposition, which is a consequence of assertion \ref{imgneg}. So,
the following 
alternative plan of the proof of assertion \ref{iwgs} could be useful for our setting; the argument can also be 
applied in other (related) contexts (this includes 
realizations mentioned in the next part of this remark).

 Theorem \ref{tmgeff}(I.2) immediately yields assertion \ref{imgneg} of Proposition \ref{pwger} 
 in the case where 
 $Y$ is a finite type $S$-scheme. 
Moreover, Theorem \ref{tmgeff}(II.2) implies that $\mgbm(X_0)[e]\in \gdos_{w\le 0}$ for any field $X_0/S$ of $\delta$-dimension $e$. Fix a finite type $X/S$ of $\delta$-dimension $d$ and consider some $\lambda_0\in L$ for it (see the proof of Proposition \ref{pcss}(I.1)). Then a (more or less) easy argument 
 (cf. the proof of Proposition 2.7.2(2) of \cite{bger}) yields that for any choice of a weight Postnikov tower $P(X)$ for $\mgbm(X)[d]$ the obvious
morphism $\prli_{\lambda\in L}P(\lambda)[d]\to P(\lambda_0)[d]$ (non-uniquely) factors through $P(X)$. 

Now let $H$ be a cohomology theory obtained via the extension procedure mentioned in Proposition \ref{pwger}(\ref{idual}). Then $H^*$ converts
homotopy limits into inductive limits; hence the inductive limits of the spectral sequences obtained via applying $H^*$ to $P(\lambda)[d]$ (i.e. the shifted
  $\delta$-coniveau spectral sequence for $H^*(X)$) is isomorphic to the spectral sequence coming from $H^*$ and $\prli_{\lambda\in L}P(\lambda)[d]$. Hence  this shifted $\delta$-coniveau spectral sequence is a retract of the one obtained via applying $H^*$ to  $P(X)$ i.e. of the corresponding {\it Gersten-weight} spectral sequence.
This immediately yields that the induced Gersten-weight filtration for $H^*(X)$ coincides with the $\delta$-coniveau one (up to a  shift by $d$). 

With somewhat more effort one can also prove that the shifted $\delta$-coniveau spectral sequence mentioned is actually isomorphic to the Gersten-weight one starting from $E_2$. Indeed, it suffices to compare the $D$-terms of the derived exact couples (for the spectral sequences in question). Now, the 'Gersten-weight' $D_2$-terms are given by the 
virtual $t$-truncations of $H^*$; the latter functors
 (also) convert
homotopy limits into inductive limits (see Remark 4.4.4(4) of \cite{bgern}). Hence one can apply an argument similar to the one used in the proof of Theorem 7.4.2 of \cite{bws} to obtain the following statement: the $D_2$-terms of our shifted $\delta$-coniveau spectral sequence calculate the corresponding values of Gersten-weight virtual $t$-truncations of $H^*$ indeed.

\item Assume that we have a {\it realization} of $\dm(-)$ (cf. \S17 of \cite{degcis}) i.e. a system of triangulated categories
$\dm'(X)$ for $X$ running through reasonable schemes (over certain fixed base scheme $S$) that is equipped with the connecting functors of the type $f^*$, $f_*$, $f^!$, and $f_!$ (cf.   
Theorem \ref{tcisdeg}) and a $2$-functor $H(-):\dm(-)\to \dm'(-)$ 'respecting' all of the connecting functors (cf. Theorem 2.4.1(I) of \cite{bmm}). Under certain (quite reasonable) conditions $H$ also yields a $2$-functor $H_{\gd}$ from $\gd(-)$ to the $2$-category $\gd'(-)$ (whose definition is similar to the definition of $\gdos$ discussed above; we take $H(\dmcs)=\dmc'(S)$ as the system of cocompact cogenerating objects of $\gd'(S)$). 

Next, we can certainly define a weight structure $w_{Ger'}(S)$ for $\gd'(S)$ that is "cogenerated by" $H(\cup_{i\in \z}\chowe(S)\{i\})$; we have  
$H_{\gd}(S)(\gdos_{w_{Ger}(S)\ge 0})\subset \gd'(S)_{w_{Ger'}(S)\ge 0}$.  For a cohomological functor $G $ from $\dmc'(S)$ to an abelian AB5 category $\au$
denote $G\circ H(S)$ by $F$. Then for any $M\in \obj \dmcs$ (or $\in \obj \gdos$)
Proposition 2.7.3(II) of \cite{bger} yields a morphism from the $\wger(S)$-weight spectral sequence converging to the  $F$-cohomolgy of $M$ to the $w_{Ger'}(S)$-one for the $G$-cohomology of $H(S)(M)$; this morphism is canonically defined starting from $E_2$  (here we denote the corresponding 'comotivic' extensions of  $G$ and $F$ by $G$ and $F$, also). In particular, we have a comparison statement for the induced 'Gersten-weight' filtrations.

Certainly, the most interesting situation is when these morphisms of spectral sequences (and filtrations) are isomorphisms (starting from $E_2$; thus, $w_{Ger'}(S)$-weight spectral sequences and filtrations can be used for the computation of the generalized $\delta$-coniveau ones). 
Using the arguments described above one can easily prove that  this is always the case if the natural analogue of Theorem \ref{tmgeff}(II.1) holds in $\dm'(S)$ (since it implies the weight-exactness of $H(S)$). The author hopes that this 
observation can shed some light on the Hodge and Tate standard conjectures (certainly, to this end one should 
find appropriate realizations for $\dm(-)$); this idea seems to be  closely related to Remark 8.14 of \cite{saitomsh}.

\item In \cite{bger} (and in \cite{bgern}) it was proved that the heart of the corresponding Gersten weight structure also contains the twisted comotif (resp. prospectrum) for an arbitrary semi-local essentially smooth affine scheme $T$ over the base field; this implies several splitting results for the comotives (resp. prospectra) of function fields and their cohomology. It seems that 
the corresponding fact for $T/S$ is wrong in general (unless 
$T$ actually lies over a field over $S$; probably, this 'negative' statement can be deduced from the main result of \cite{aycex}). This is one of the reasons for the author for not giving the full proofs of the claims of this subsection.

\item One can also consider the Chow weight structure on $\gdos$; this is the weight structure "cogenerated by" $\chows\subset \dmcs\subset \gdos$. We will call the corresponding spectral sequences and filtrations {\it Chow-weight} ones (as we did in \cite{brelmot}; cf. also \S4.7 of \cite{bger}).

In order to compare this weight structure with the Gersten one, one should 
restrict both of them to the "effective" subcategory $\gdeos$ of $\gdos$ i.e. to the subcategory cogenerated by $\chowes$ (cf. Proposition \ref{pdme}). 
For the corresponding restrictions we obviously have $\gdeos_{\wchow(S)\ge 0}\subset \gdeos_{\wchow(S)\ge 0}$. 
Hence one can easily carry over the results of \S4.8 of \cite{bger} to this setting. In particular, for any $M\in \obj \gdeos$ and a cohomological $H:\gdeos\to \au$ there exist a morphism from the $\wchow(S)$-weight spectral sequence for $(H,M)$ to the corresponding $\wger(S)$-one (i.e. to the generalized $\delta$-coniveau one). In particular, one can take $M=\mgs(X)(s)$ for a smooth $X/S$ here, where $s$ is the 
 $\delta$-dimension of $S$.

\end{enumerate}
\end{rema}


\end{document}